\documentclass[reqno, a4paper] {article}
[10pt]

 \usepackage[top=2.5cm, bottom=3.5cm, left=3.5cm, right= 3.5cm]{geometry}

\makeatletter
\renewcommand*\l@section{\@dottedtocline{1}{1.5em}{2.3em}}
\makeatother

\usepackage{amsfonts}
\usepackage{amssymb}
\usepackage[T1]{fontenc}

\usepackage{tikz}
\usetikzlibrary{calc}

\usepackage{CJK}
\usepackage{amsmath}
 \usepackage{algorithmicx}
 \usepackage[ruled]{algorithm}
 \usepackage{algpseudocode}
 
\usepackage{amsfonts}
\usepackage{amssymb}
\usepackage{amsthm}
\usepackage{amssymb}
\usepackage{enumerate}
\usepackage[calc]{picture}
\usepackage[all,cmtip]{xy}

\usepackage[mathscr]{eucal}
\usepackage{eqlist}

\usepackage{color}
\usepackage{abstract} 
\usepackage[T1]{fontenc}
 
\theoremstyle{plain}
\newtheorem{theorem}{Theorem}
\newtheorem{proposition}[theorem]{Proposition}
\newtheorem{lemma}[theorem]{Lemma}

\newtheorem{example}[theorem]{Example}
\newtheorem{corollary}[theorem]{Corollary}

\theoremstyle{definition}
\newtheorem{definition}{Definition}
  
\usepackage{etoolbox}
\newtheoremstyle{myrem}
 {3pt}
 {3pt}
 {\normalsize}
 { }
 {\itshape}
 {:}
 { }
 {}

 \theoremstyle{myrem}
 \newtheorem{remark}{Remark}
 \appto\remark{\leftskip\parindent}
 \appto\remark{\rightskip\parindent}

\numberwithin{equation}{section}
\numberwithin{theorem}{section}

\begin{document}

\begin{center}
{\Large {\textbf {Hodge Decompositions  for Weighted Hypergraphs  }}}
 \vspace{0.58cm}

Shiquan Ren*, Chengyuan Wu,  Jie Wu
     
\bigskip

\footnotetext{

\rule{2.4mm}{0mm}{\bf Keywords and Phrases.}  weighted hypergraphs, weighted simplicial complexes, cohomology, Laplacian. 

}

\bigskip

\parbox{24cc}{{\small

{\textbf{Abstract}.}  

Weighted hypergraphs are generalizations of  weighted simplicial complexes. In recent years,    weighted Laplacians of weighted simplicial complexes have been studied.  In 2016,  as a generalization of the homology of simplicial complexes, the embedded homology of hypergraphs was constructed. In this paper, we generalize the weighted Laplacians of weighted simplicial complexes to weighted hypergraphs. We study the relations between the weighted Laplacians and the weighted embedded  homology of weighted hypergraphs. We generalize the Hodge decompositions of weighted simplicial complexes to weighted hypergraphs.  Moreover, as a complement for the Hodge decompositions, we give some results for the nonzero eigenvalues of the weighted Laplacians of weighted hypergraphs. 
}}
\end{center}

\bigskip


\section{Introduction}\label{s1}

Hypergraphs  (cf. \cite{berge}) are  higher-dimensional generalizations of graphs. In a graph, an edge is a segment joining two vertices, which is of dimension $1$. While in a hypergraph, an $n$-dimensional hyperedge (or simply an $n$-hyperedge) is a set  of   $n+1$ vertices. 

\begin{definition}[Hypergraph]\cite{berge,parks}
A {hypergraph}  is a pair $(V_\mathcal{H},\mathcal{H})$ where $V_\mathcal{H}$ is a set   and $\mathcal{H}$ is a subset of  the power set of $V_\mathcal{H}$.  We call an element of $V_\mathcal{H}$  a {vertex} and call an element of $\mathcal{H}$   a {hyperedge}. For any $n\geq 0$, we call a hyperedge consisting of $n+1$ vertices   an {$n$-hyperedge}. We call a nonempty subset of a hyperedge as a face of the hyperedge. 
\end{definition}
In this paper, we  assume that each hyperedge contains at least one vertex.  We also assume that each vertex in $V_\mathcal{H}$ appears in at least one hyperedge of $\mathcal{H}$. Then $V_\mathcal{H}$ is the union of all the vertices of the hyperedges of $\mathcal{H}$.  Hence we can simply denote the hypergraph $(V_\mathcal{H},\mathcal{H})$ as $\mathcal{H}$.

\smallskip

(Abstract) simplicial complexes   can be regarded as special hypergraphs such that all the faces of hyperedges are still hyperedges. 
\begin{definition}[(Abstract) Simplicial Complex]
Let $\mathcal{H}$ be a hypergraph. If for any hyperedge $\sigma\in\mathcal{H}$ and any nonempty subset $\tau\subseteq \sigma$, we always have $\tau\in\mathcal{H}$, then $\mathcal{H}$ is called an (abstract) simplicial complex.  In this case, $\mathcal{H}$ is denoted as $\mathcal{K}$, and the hyperedges are called simplices.   
 \end{definition}

The graph Laplacian is a self-adjoint operator on graphs defined by the adjacency relations of the vertices (cf. \cite[Section~1.2]{chungbook1}).  In 1847, the  graph Laplacian   was firstly investigated by G. Kirchhoff \cite{1847} in the study of electrical networks.   Since 1970's, the spectrum of the graph Laplacian  has been  extensively investigated (cf. \cite{anderson,ban,chungbook1,cve}).

  The Laplacian of simplicial complexes is a generalization of the graph Laplacian to higher dimensions.  
A simplicial complex $\mathcal{K}$ has an associated chain complex  $C_n(\mathcal{K})$, $n\geq 0$,  with boundary maps $\partial_n: C_{n}(\mathcal{K})\longrightarrow C_{n-1}(\mathcal{K})$  such that $\partial_{n}\partial_{n+1}=0$.  We  construct the Laplacian of simplicial complexes as (cf. \cite{eck}, \cite[p. 4314]{duv}, \cite[p. 304]{adv1}) 
\begin{eqnarray}\label{eq001}
L_n=\partial_{n+1}\partial^*_{n+1}+\partial_n^*\partial_n. 
\end{eqnarray}
Here $\partial_n^*$ (respectively $\partial_{n+1}^*$) is the dual operator of $\partial_n$ (respectively $\partial_{n+1}$) with respect to certain inner product on each $C_*(\mathcal{K})$, $*\geq 0$. 
In 1944, a discrete version of the Hodge theorem for $L_n$ was proved by Eckmann \cite{eck} (cf. \cite[Theorem~3.3]{duv}, \cite[Theorem~2.2]{adv1}).   
In 2002, the spectrum of the Laplacian $L_n$ was investigated by A.M. Duval and V. Reiner \cite{duv}.

Weighted simplicial complexes are simplicial complexes equipped with certain weight functions on the simplices. In 1990, by twisting the boundary maps  using the weights, R.J. MacG. Dawson \cite{daw}  studied the homology of weighted simplicial complexes.  
 In 2013, by twisting the boundary maps in the Laplacians   (\ref{eq001}) using the weights, and considering the cohomology,  D. Horak and J. Jost \cite{glob1,adv1} studied the weighted Laplacians of weighted simplicial complexes.  Recently, the weight functions on simplicial complexes were generalized to inner products on cochain complexes by C. Wu, S. Ren, J. Wu and K. Xia \cite{chengyuan}.  The properties, classifications and applications   of weighted (co)homology and weighted Laplacians of weighted simplicial complexes were studied in \cite{rocky,chengyuan2,chengyuan}. 
 
 \smallskip
 
On the other hand, in order to investigate the topology of hypergraphs, some homology groups have been considered.   In 1991,   by adding all the missing faces of $\mathcal{H}$, the associated simplicial complex $\Delta\mathcal{H}$ of $\mathcal{H}$ was defined by A.D. Parks and S.L. Lipscomb \cite{parks}.   The homology groups of $\Delta\mathcal{H}$ were studied to investigate the topology of $\mathcal{H}$. 
In 2016,  homology of simplicial complexes was generalized to the embedded homology of hypergraphs   by S. Bressan, J. Li, S. Ren and J. Wu \cite{h1}.     The original idea of the embedded homology was given by A. Grigor'yan, Y. Lin, Y. Muranov and S.T. Yau \cite{yau1} in the study of  paths of digraphs.  For a general hypergraph $\mathcal{H}$, the embedded homology of $\mathcal{H}$ and the homology of $\Delta\mathcal{H}$ are not isomorphic. They reflect different aspects of the topology of $\mathcal{H}$.  Let $\mathbb{F}$ be the real numbers $\mathbb{R}$ or the complex numbers $\mathbb{C}$.   The embedded homology of $\mathcal{H}$ with coefficients in $\mathbb{F}$ is denoted by $H_n(\mathcal{H};\mathbb{F})$, $n\geq 0$. 
 
 \smallskip
 
 In this paper,  we generalize the weighted (co)homology and the weighted Laplacian studied  in \cite{daw,glob1,adv1,rocky,chengyuan2,chengyuan} from   weighted simplicial complexes to  weighted hypergraphs and prove a Hodge decomposition for weighted hypergraphs.  A weighted hypergraph $(\mathcal{H},\phi)$ is a hypergraph $\mathcal{H}$ equipped with a weight $\phi$ on the associated simplicial complex $\Delta\mathcal{H}$ (cf. Definition~\ref{def1} in  Subsection~\ref{subs4.1}). We denote the weighted  Laplacian of the weighted simplicial complex  $(\Delta\mathcal{H},\phi)$ as $L_n^{\Delta\mathcal{H},\phi}$.  Then the kernel of $L_n^{\Delta\mathcal{H},\phi}$ is linearly isomorphic to the weighted homology $H_n(\Delta\mathcal{H},\phi;\mathbb{F})$ (cf. \cite{chengyuan}).  
  We generalize the embedded homology  $H_n(\mathcal{H};\mathbb{F})$ (cf. \cite[Subsection 3.2]{h1})  to the weighted embedded homology $H_n(\mathcal{H},\phi;\mathbb{F})$.    We generalize the infimum chain complex $\text{Inf}_*(\mathcal{H})$ and the supremum chain complex $\text{Sup}_*(\mathcal{H})$ (cf.  \cite[Proposition~3.3]{h1})  to the weighted infimum chain complex $\text{Inf}^\phi_*(\mathcal{H})$ and the weighted supremum chain complex $\text{Sup}^\phi_*(\mathcal{H})$. We generalize the weighted  Laplacian of weighted simplicial complexes (cf. \cite[Definition~2.1]{adv1}) to the weighted infimum Laplacian $L_n^{\text{Inf}^\phi_*(\mathcal{H}),\phi}$ and the weighted supremum Laplacian $L_n^{\text{Sup}^\phi_*(\mathcal{H}),\phi}$. 
   Then both the kernel  of   $L_n^{\text{Inf}^\phi_*(\mathcal{H}),\phi}$ and the kernel of $L_n^{\text{Sup}^\phi_*(\mathcal{H}),\phi}$
    are  linearly isomorphic to $H_n(\mathcal{H},\phi;\mathbb{F})$ (cf. Theorem~\ref{pr.a.1.w}).  
 The main result of this paper is the next theorem. 
 
  \begin{theorem}[Theorem~\ref{th-4.19}]\label{th-0.0}
Let $\mathcal{H}$ be a hypergraph, $\phi$ a weight on $\mathcal{H}$, and $n\geq 0$. Let $s$ be the canonical inclusion from $\mathcal{H}$ to $\Delta\mathcal{H}$ and $s_*$ be the induced homomorphism from $H_n(\mathcal{H},\phi;\mathbb{F})$ to $H_n(\Delta\mathcal{H},\phi;\mathbb{F})$. Then  represented by the kernel of the weighted  supremum Laplacian $\text{Ker}(L_n^{\text{Sup}^\phi_*(\mathcal{H}),\phi})$, $H_n(\mathcal{H},\phi;\mathbb{F})$ is the orthogonal sum of  $\text{Ker}(L_n^{\Delta\mathcal{H},\phi})\cap\text{Inf}^\phi_n(\mathcal{H})$ and $\text{Ker}(s_*)$.  And represented by the kernel of the weighted  Laplacian $\text{Ker}(L_n^{\Delta\mathcal{H},\phi})$,  $H_n(\Delta\mathcal{H},\phi;\mathbb{F})$ is the orthogonal sum of 
$\text{Ker}(L_n^{\Delta\mathcal{H},\phi})\cap\text{Inf}^\phi_n(\mathcal{H})$ and $\text{Coker}(s_*)$.  \end{theorem}

In Section~\ref{sss2},   we  study the Hodge isomorphisms   for  hypergraphs by using the embedded homology.  In Section~\ref{s-a.3}, we  study the Hodge decompositions for  hypergraphs by using the embedded homology as well as the homology of associated complexes.  We prove  Theorem~\ref{th-0.0} (Theorem~\ref{th-4.19}) for the particular case of hypergraphs (hypergraphs can be regarded as weighted hypergraphs with trivial weight) in Theorem~\ref{th-3.19}.  In Section~\ref{sss4}, we generalize the Hodge isomorphisms in Section~\ref{sss2} and the Hodge decompositions in Section~\ref{s-a.3} from hypergraphs to weighted hypergraphs.  We generalize Theorem~\ref{th-3.19} from hypergraphs to weighted hypergraphs and  obtain  Thteorem~\ref{th-0.0} (Theorem~\ref{th-4.19}).   
In Section~\ref{sec-a}, as a complement for the Hodge decompositions,  we study the nonzero eigenvalues of the weighted Laplacians for   weighted hypergraphs.  In Section~\ref{sec6}, we discuss some relations between hypergraphs and paths on digraphs, which provide a potential motivation for this paper.

\smallskip

Besides the Laplacians on hypergraphs and the weighted Laplacians on weighted hypergraphs considered in this paper,  there are   other kinds of Laplacians for hypergraphs.  For example, in 1983, the graph Laplacian was generalized to  certain Laplacians of hypergraphs by F.R.K. Chung  \cite{chung123}.  And in 2015,  S. Hu and L. Qi \cite{huqi} constructed certain Laplacians for uniform hypergraphs.  Moreover, in \cite{lio}, some $p$-Laplacians were constructed for submodular hypergraphs.  Our Laplacians on hypergraphs have an advantage that it gives a natural connection with the embedded homology of hypergraphs and it induces Hodge decomposition theorems. The connections between the embedded homology of hypergraphs and other type Laplacians need to be explored further.

\smallskip

Throughout this paper, we assume that hypergraphs (respectively, simplicial complexes) have finitely many hyperedges (respectively, simplices).  

{\it

}

 \section{Hodge Isomorphisms for Hypergraphs}\label{sss2}

In this section, we generalize the Hodge isomorphism  from simplicial complexes to hypergraphs.

\subsection{Proof of Theorem~\ref{pr.a.1}}

In this subsection, we prove the first part of the Hodge isomorphism for hypergraphs in Theorem~\ref{pr.a.1}. 

\smallskip

Let $\mathcal{H}$ be a hypergraph.   
 The  associated complex $\Delta\mathcal{H}$ of $\mathcal{H}$ is the smallest simplicial complex that $\mathcal{H}$ can be embedded in (cf. \cite{parks}). It  consists of the simplices (cf. \cite[Section~3.1]{h1}, \cite[Section~2.1]{ev})
\begin{eqnarray*}
\Delta \mathcal{H}=\{\eta\subseteq\tau\mid \tau\in \mathcal{H}\}.
\end{eqnarray*}
 Let $n$ be  a nonnegative integer.   Let $\mathbb{F}(\Delta\mathcal{H})_n$ be the vector space over $\mathbb{F}$ with basis all the $n$-simplices of $\Delta\mathcal{H}$. 
We have a chain complex
\begin{eqnarray*}
0\overset{\partial_{1}}{\longleftarrow} \mathbb{F}(\Delta\mathcal{H})_1 \overset{\partial_{2}}{\longleftarrow}\mathbb{F}(\Delta\mathcal{H})_2\overset{\partial_{3}}{\longleftarrow}\cdots \overset{\partial_{n}}{\longleftarrow}\mathbb{F}(\Delta\mathcal{H})_n\overset{\partial_{n+1}}{\longleftarrow}\mathbb{F}(\Delta\mathcal{H})_{n+1}\overset{\partial_{n+2}}{\longleftarrow}\cdots 
\end{eqnarray*}
We denote the chain complex   as $(\mathbb{F}(\Delta\mathcal{H})_*,\partial_*)$. Let $\langle~,~\rangle$ be the canonical real or complex inner product on $\mathbb{F}(\Delta\mathcal{H})_n$ given by
\begin{eqnarray*}
&\langle \sum_i a_i\sigma_i,\sum_j b_j \tau_j\rangle= \sum_{\sigma_i=\tau_j}a_i\bar b_j. 
\end{eqnarray*}
Here $\sigma_i,\tau_j\in \Delta\mathcal{H}$, $\dim\sigma_i=\dim\tau_j=n$, and $a_i,b_j\in\mathbb{F}$. 
The number $\bar b_j$ is $b_j$ if $\mathbb{F}=\mathbb{R}$, and $\bar b_j$   is the complex conjugate of $b_j$ if $\mathbb{F}=\mathbb{C}$.  The adjoint of $\partial_n$ is a linear map 
\begin{eqnarray*}
\partial_n^*: \mathbb{F}(\Delta\mathcal{H})_{n-1}\longrightarrow\mathbb{F}(\Delta\mathcal{H})_n
\end{eqnarray*}
 such that  
\begin{eqnarray}\label{eq-1.a}
\langle \partial_n \omega, \omega'\rangle =\langle \omega,\partial^*_n\omega'\rangle
\end{eqnarray}
for any $\omega\in \mathbb{F}(\Delta\mathcal{H})_n$ and any $\omega'\in \mathbb{F}(\Delta\mathcal{H})_{n-1}$.  Equivalently,  (\ref{eq-1.a}) can be written as 
\begin{eqnarray}\label{eq-2.a}
\langle \partial_n \sigma, \tau\rangle =\langle \sigma,\partial^*_n\tau\rangle
\end{eqnarray}
for any $\sigma,\tau\in \Delta\mathcal{H}$ with $\dim\sigma=n$ and $\dim\tau=n-1$.  The matrix of $\partial_n^*$ is the conjugate transpose of the matrix of $\partial_n$, i.e.
\begin{eqnarray*}
 [\partial_n^*]=\overline {[\partial_n]}^{T}, 
 \end{eqnarray*}
 under any orthonormal basis of $\mathbb{F}(\Delta\mathcal{H})_n$ and any orthonormal basis of $\mathbb{F}(\Delta\mathcal{H})_{n-1}$.

By   \cite{eck}, \cite[p. 4314]{duv} and \cite[p. 304]{adv1}, 
we define the combinatorial Laplacian of $(\mathbb{F}(\Delta\mathcal{H})_*,\partial_*)$ as  
\begin{eqnarray*}
L_n^{\Delta\mathcal{H}}=\partial_{n+1}\partial^*_{n+1}+\partial^*_n\partial_n. 
\end{eqnarray*}
We notice that for any $\omega\in \mathbb{F}(\mathcal{H})_n$, 
\begin{eqnarray*}
\langle L_n^{\Delta\mathcal{H}} \omega,\omega\rangle &=&\langle\partial_{n+1}\partial_{n+1}^*\omega,\omega\rangle+\langle \partial^*_n\partial_n\omega,\omega\rangle\nonumber\\
&=&\langle \partial^*_{n+1}\omega,\partial^*_{n+1}\omega\rangle+\langle \partial_n\omega,\partial_n\omega\rangle. 
\label{eq2.777}
\end{eqnarray*}
Hence $L^{\Delta\mathcal{H}}_n\omega=0 $   if and only if 
$\partial_n\omega=\partial^*_{n+1}\omega=0$. 
Therefore,
\begin{eqnarray}\label{eq-a.2}
\text{Ker}L^{\Delta\mathcal{H}}_n=\text{Ker}\partial_n \cap \text{Ker} \partial_{n+1}^*.  
\end{eqnarray}
 By (\ref{eq-a.2}) and the Hodge isomorphism of simplicial complexes (cf. \cite{eck}),  
  \begin{eqnarray}\label{eq-a.91}
 H_n(\Delta\mathcal{H})\cong \text{Ker}L^{\Delta\mathcal{H}}_n
 \cong \text{Ker}\partial_n \cap \text{Ker} \partial_{n+1}^*. 
 \end{eqnarray}
 Since $\text{Ker} \partial_{n+1}^*= (\text{Im}\partial_{n+1})^\perp$,  (\ref{eq-a.91}) can be written in terms of $\partial_*$ as 
 \begin{eqnarray*} 
  H_n(\Delta\mathcal{H})\cong \text{Ker}\partial_n \cap (\text{Im}\partial_{n+1})^\perp. 
 \end{eqnarray*}

Let $\mathbb{F}(\mathcal{H})_n$ be the vector space over $\mathbb{F}$ with basis all the $n$-hyperedges of $\mathcal{H}$. By \cite[Section~2 and Section~3]{h1},  the infimum chain complex and the  supremum chain complex  of $\mathcal{H}$ are respectively
\begin{eqnarray*}
\text{Inf}_n(\mathcal{H})&=&\mathbb{F}(\mathcal{H})_n \cap \partial^{-1}_{n} \mathbb{F}(\mathcal{H})_{n-1},\\
\text{Sup}_n(\mathcal{H})&=&\mathbb{F}(\mathcal{H})_n + \partial^{}_{n+1} \mathbb{F}(\mathcal{H})_{n+1}.
\end{eqnarray*}
By restricting $\partial_n$ to $\text{Inf}_n(\mathcal{H})$ and $\text{Sup}_n(\mathcal{H})$ respectively, we obtain the boundary maps  
\begin{eqnarray*}
\partial_n\mid_{\text{Inf}_*(\mathcal{H})}:  \text{Inf}_n(\mathcal{H})\longrightarrow \text{Inf}_{n-1}(\mathcal{H})
\end{eqnarray*}
 of  the chain complex $\text{Inf}_*(\mathcal{H})$ and the boundary maps
 \begin{eqnarray*}
 \partial_n\mid_{\text{Sup}_*(\mathcal{H})}:  \text{Sup}_n(\mathcal{H})\longrightarrow \text{Sup}_{n-1}(\mathcal{H})
 \end{eqnarray*}
  of the chain complex $\text{Sup}_*(\mathcal{H})$.  We have a commutative diagram of real or complex Euclidean spaces and linear maps
\begin{eqnarray*}
\xymatrix{
\text{Inf}_{n+1}(\mathcal{H})\ar@/_/[dd]_{\partial_{n+1}\mid_{\text{Inf}_*(\mathcal{H})}}\ar[rr] &&\mathbb{F}(\mathcal{H})_{n+1} 
\ar[rr] && \text{Sup}_{n+1}(\mathcal{H})\ar[rr] \ar@/_/[dd]_{\partial_{n+1}\mid_{\text{Sup}_*(\mathcal{H})}}&& \mathbb{F}(\Delta\mathcal{H})_{n+1}\ar@/_/[dd]_{\partial_{n+1}}\\
\\
\text{Inf}_{n }(\mathcal{H}) \ar@/_/[uu]_{{(\partial_{n+1}\mid_{\text{Inf}_*(\mathcal{H})})^*}} \ar[rr] &&\mathbb{F}(\mathcal{H})_{n } \ar[rr] && \text{Sup}_{n }(\mathcal{H})\ar[rr] \ar@/_/[uu]_{(\partial_{n+1}\mid_{\text{Sup}_*(\mathcal{H})})^*} && \mathbb{F}(\Delta\mathcal{H})_{n }\ar@/_/[uu]_{\partial_{n+1}^*}.
}
\end{eqnarray*}
Here the horizontal maps are the canonical inclusions.  

We define the supremum Laplacian and the infimum Laplacian of $\mathcal{H}$ respectively as 
\begin{eqnarray*}
L_n^{\text{Inf}_*(\mathcal{H})}&=&(\partial_{n+1}\mid_{\text{Inf}_*(\mathcal{H})})(\partial_{n+1}\mid_{\text{Inf}_*(\mathcal{H})})^* +(\partial_n\mid_{\text{Inf}_*(\mathcal{H})})^*(\partial_n\mid_{\text{Inf}_*(\mathcal{H})}),\\
L_n^{\text{Sup}_*(\mathcal{H})}&=&(\partial_{n+1}\mid_{\text{Sup}_*(\mathcal{H})})(\partial_{n+1}\mid_{\text{Sup}_*(\mathcal{H})})^* +(\partial_n\mid_{\text{Sup}_*(\mathcal{H})})^*(\partial_n\mid_{\text{Sup}_*(\mathcal{H})}). 
\end{eqnarray*}
Then similar to (\ref{eq-a.2}), 
\begin{eqnarray}
\text{Ker}(L_n^{\text{Inf}_*(\mathcal{H})})&=&\text{Ker}(\partial_n\mid_{\text{Inf}_*(\mathcal{H})})\cap \text{Ker}(\partial_{n+1}\mid_{\text{Inf}_*(\mathcal{H})})^*,\label{eq-a.01}\\
\text{Ker}(L_n^{\text{Sup}_*(\mathcal{H})})&=&\text{Ker}(\partial_n\mid_{\text{Sup}_*(\mathcal{H})})\cap \text{Ker}(\partial_{n+1}\mid_{\text{Sup}_*(\mathcal{H})})^*. \label{eq-a.02}
\end{eqnarray}
The next theorem proves that the kernels of the Laplacians are isomorphic to the embedded homology of hypergraphs.  

\begin{theorem}[Hodge Isomorphism for Hypergraphs: Part I]\label{pr.a.1}
Let $\mathcal{H}$ be a hypergraph.  For each $n\geq 0$, 
\begin{eqnarray*}
H_n(\mathcal{H};\mathbb{F}) \cong   \text{Ker}(L_n^{\text{Inf}_*(\mathcal{H})})  
 \cong  \text{Ker}(L_n^{\text{Sup}_*(\mathcal{H})}). 
\end{eqnarray*}
In other words, 
\begin{eqnarray*}
H_n(\mathcal{H};\mathbb{F})&\cong&  \text{Ker}(\partial_n\mid_{\text{Inf}_*(\mathcal{H})})\cap \text{Ker}(\partial_{n+1}\mid_{\text{Inf}_*(\mathcal{H})})^* \\
 &\cong& \text{Ker}(\partial_n\mid_{\text{Sup}_*(\mathcal{H})})\cap \text{Ker}(\partial_{n+1}\mid_{\text{Sup}_*(\mathcal{H})})^*.
\end{eqnarray*}
\end{theorem}

Before proving Theorem~\ref{pr.a.1}, we give the following lemma. 

\begin{lemma}[Hodge  Isomorphism of Chain Complexes]\label{le-a.11}
Let $C_*$ be a graded vector space  over $\mathbb{F}$.  Suppose for each $n\geq 0$, there are maps $d_{n+1}: C_{n+1}\longrightarrow C_{n}$ such that $d_{n+1}d_n=0$.  Let $L_n=d_{n+1}d_{n+1}^*+ d_n^* d_n$.  Then the homology
$H_n(\{C_*,d_*\})$ of the chain complex $\{C_*,d_*\}$ is isomorphic to $\text{Ker} L_n$.  
\end{lemma}

\begin{proof}
Lemma~\ref{le-a.11} is an analogue of \cite[Theorem~4.16]{morita} for chain complexes.  With minor modifications, the proof of \cite[Theorem~4.16]{morita} applies. 
\end{proof}

Now we prove Theorem~\ref{pr.a.1}. 
\begin{proof}[Proof of Theorem~\ref{pr.a.1}]
By Lemma~\ref{le-a.11}, we have
\begin{eqnarray*}
H_n(\text{Inf}_*(\mathcal{H}))&\cong& \text{Ker}(L_n^{\text{Inf}_*(\mathcal{H})}), \\
H_n(\text{Sup}_*(\mathcal{H}))&\cong& \text{Ker}(L_n^{\text{Sup}_*(\mathcal{H})}). 
\end{eqnarray*} 
By \cite[Proposition~3.4]{h1}, the embedded homology of $\mathcal{H}$ is given by
\begin{eqnarray*}
H_n(\mathcal{H};\mathbb{F})\cong H_n(\text{Inf}_*(\mathcal{H}))\cong H_n(\text{Sup}_*(\mathcal{H})). 
\end{eqnarray*}
The assertion follows. 
\end{proof}

\subsection{Proof of Theorem~\ref{th-a.1}}

In this subsection, we prove the second part of the Hodge isomorphism for hypergraphs in Theorem~\ref{th-a.1}. 

\smallskip

For a (real or complex) Euclidean space $V$ and a subspace $W$ in $V$, let $\perp(W,V)$ be the orthogonal complement of $W$ in $V$.  
As graded vector spaces, we have
\begin{eqnarray*}
\text{Inf}_n(\mathcal{H})\subseteq \mathbb{F}(\mathcal{H})_n\subseteq  \text{Sup}_n(\mathcal{H})\subseteq \mathbb{F}(\Delta\mathcal{H})_n. 
\end{eqnarray*}
Hence we have the orthogonal decompositions 
\begin{eqnarray*}
\mathbb{F}(\mathcal{H})_n&=&\text{Inf}_n(\mathcal{H})\oplus A_n,\\
\text{Sup}_n(\mathcal{H})&=&\mathbb{F}(\mathcal{H})_n\oplus B_n,\\
\mathbb{F}(\Delta\mathcal{H})_n&=&\text{Sup}_n(\mathcal{H}) \oplus E_n.  
\end{eqnarray*}
Here  $A_n$, $B_n$ and $D_n$ are subspaces of $\mathbb{F}(\Delta\mathcal{H})_n$ given by
\begin{eqnarray*}
A_n&=&\perp \big(\mathbb{F}(\mathcal{H})_n\cap \partial_n^{-1}\mathbb{F}(\mathcal{H})_{n-1}, \mathbb{F}(\mathcal{H})_n\big),\\
B_n&=&\perp \big(\mathbb{F}(\mathcal{H})_n,\mathbb{F}(\mathcal{H})_n+ \partial_{n+1} \mathbb{F}(\mathcal{H})_{n+1}\big),\\
E_n&=&\perp \big(\mathbb{F}(\mathcal{H})_{n+1}+ \partial_{n+1} \mathbb{F}(\mathcal{H})_{n+1},\mathbb{F}(\Delta\mathcal{H})_n\big).  
\end{eqnarray*}  
If we consider the complement hypergraph $\mathcal{H}^c=\Delta\mathcal{H}\setminus \mathcal{H}$ (cf.  \cite[Subsection~3.1]{h1}),  then 
\begin{eqnarray*}
B_n\oplus E_n= \mathbb{F}(\mathcal{H}^c)_n. 
\end{eqnarray*}

\begin{proposition}\label{le-a.1}
Let $\mathcal{H}$ be a hypergraph and $n\geq 0$. Then  
\begin{eqnarray}\label{eq-a.3}
\text{Ker} (\partial_n^*\mid _{\text{Inf}_{*}(\mathcal{H})})\subseteq  \text{Ker} (\partial_n\mid _{\text{Inf}_*(\mathcal{H})})^* 
\end{eqnarray}
and 
\begin{eqnarray}\label{eq-a.4}
\text{Ker} (\partial_n^*\mid _{\text{Sup}_{*}(\mathcal{H})})\subseteq  \text{Ker} (\partial_n\mid _{\text{Sup}_*(\mathcal{H})})^*.
\end{eqnarray}
Moreover,
\begin{enumerate}[(a).]
\item
{If} 
$\partial_n(A_n\oplus B_n\oplus E_n)\subseteq A_{n-1}\oplus B_{n-1}\oplus E_{n-1}$, then  $\partial_n^*\mid _{\text{Inf}_{n-1}(\mathcal{H})}=(\partial_n\mid _{\text{Inf}_*(\mathcal{H})})^*$, and  the equality in (\ref{eq-a.3}) holds; 
\item {If} 
$\partial_n( E_n)\subseteq E_{n-1}$, then $\partial_n^*\mid _{\text{Sup}_{n-1}(\mathcal{H})} =   (\partial_n\mid _{\text{Sup}_*(\mathcal{H})})^*$, and the equality in (\ref{eq-a.4}) holds.
\end{enumerate}
\end{proposition}

Before proving Proposition~\ref{le-a.1},  we prove the next lemma.

\begin{lemma}\label{le-linearalg}
Let $W$ and $W'$ be real or complex Euclidean spaces with inner products $\langle~,~\rangle$ and $\langle~,~\rangle'$ respectively.  Let $T: W\longrightarrow W'$ be a linear map.  Let $V$  and $V'$ be subspaces of $W$ and $W'$ respectively such that $TV\subseteq V'$. Let $\perp(V,W)$ and $\perp(V',W')$ be the orthogonal complements of $V$ in $W$ and of $V'$ in $W'$ respectively.  Then the diagram commutes
\begin{eqnarray*}
\xymatrix{
V'\ar[rr]^{T^*\mid _{V'}}\ar[rrdd]_{(T\mid_V)^*}&& W=V\oplus \perp(V,W)\ar[dd]^{\text{orthogonal proj.}}\\
\\
&& V. 
}
\end{eqnarray*}
Here $(-)^*$ denotes the adjoint of a linear map.  In particular, \begin{eqnarray}\label{eq-a.8}
 \text{Ker}(T^*\mid _{V'})  \subseteq\text{Ker}(T\mid_V)^*.
 \end{eqnarray}
 Moreover, if  
$T\big(\perp(V,W)\big)\subseteq \perp(V',W')$, then $T^*\mid _{V'}=(T\mid_V)^*$  
and the equality in (\ref{eq-a.8}) holds. 
\end{lemma}
\begin{proof}
Let $e_1,\ldots,e_k$ be an orthonormal basis of $V$. We extend it to be an orthonormal basis $e_1,\ldots,e_k,e_{k+1},\ldots, e_n$ of $W$.  Let $e'_1,\ldots,e'_t$ be an orthonormal basis of $V'$. We extend it to be an orthonormal basis  $e'_1,\ldots,e'_t, e'_{t+1},\ldots, e'_m$ of $W'$.  Let $1\leq j\leq t$. Then 
\begin{eqnarray}
(T^*\mid _{V'})e'_j&=&\sum_{i=1}^n \langle e_i, (T^*\mid _{V'})e'_j\rangle e_i\nonumber\\
&=&\sum_{i=1}^n \langle e_i, T^* e'_j\rangle e_i\nonumber\\
&=& \sum_{i=1}^n \langle Te_i, e'_j\rangle e_i. \label{eq-com.1}
\end{eqnarray}
And 
\begin{eqnarray}
(T\mid _V)^* e'_j &=& \sum_{i=1}^k \langle e_i,(T\mid _V)^* e'_j\rangle e_i\nonumber\\
&=& \sum_{i=1}^k \langle (T\mid _V)  e_i,e'_j\rangle e_i\nonumber\\
&=&\sum_{i=1}^k \langle T e_i,e'_j\rangle e_i. \label{eq-com.2}
\end{eqnarray}
The commutative diagram follows from (\ref{eq-com.1}) and (\ref{eq-com.2}).  For any $v'\in V'$, if $(T^*\mid _{V'})v'=0$, then by the commutative diagram, $(T\mid _V)^* v'=0$. Hence  (\ref{eq-a.8}) follows.

Suppose $T(\perp(V,W))\subseteq \perp(V',W')$.  Then  for any $v'\in V'$ and any $v^\perp\in \perp(V,W)$, 
\begin{eqnarray*}
\langle (T^*\mid_{V'})v',v^\perp\rangle=\langle v',T(v^\perp)\rangle=0. 
\end{eqnarray*}
Hence $(T^*\mid_{V'})v'\in V$.   By (\ref{eq-com.1}) and (\ref{eq-com.2}), we have $(T^*\mid _{V'})v'=(T\mid_V)^*v'$.  Hence  $T^*\mid _{V'}=(T\mid_V)^* $.  Therefore, the equality in (\ref{eq-a.8}) holds.   
\end{proof}

Now we prove Proposition~\ref{le-a.1}. 

\begin{proof}[Proof of Proposition~\ref{le-a.1}] 
In Lemma~\ref{le-linearalg}, let $W$ be $\mathbb{F}(\Delta\mathcal{H})_n$ and $W'$ be $\mathbb{F}(\Delta\mathcal{H})_{n-1}$. Let $T$ be $\partial_n$. 

(a). Let $V$ be $\text{Inf}_n(\mathcal{H})$ and $V'$ be $\text{Inf}_{n-1}(\mathcal{H})$ in Lemma~\ref{le-linearalg}. 
Since 
\begin{eqnarray*}
\perp\big(\text{Inf}_n(\mathcal{H}),\mathbb{F}(\Delta\mathcal{H})_n\big)&=& A_n\oplus B_n\oplus E_n,\\
\perp\big(\text{Inf}_{n-1}(\mathcal{H}),\mathbb{F}(\Delta\mathcal{H})_{n-1}\big)&=& A_{n-1}\oplus B_{n-1}\oplus E_{n-1},
\end{eqnarray*}
we have (\ref{eq-a.3}). Moreover, if $\partial_n(A_n\oplus B_n\oplus E_n)\subseteq A_{n-1}\oplus B_{n-1}\oplus E_{n-1}$, then $\partial_n^*\mid _{\text{Inf}_{n-1}(\mathcal{H})}=(\partial_n\mid _{\text{Inf}_n(\mathcal{H})})^*$, and the equality in (\ref{eq-a.3}) holds. 

(b). Let $V$ be $\text{Sup}_n(\mathcal{H})$ and $V'$ be $\text{Sup}_{n-1}(\mathcal{H})$ in Lemma~\ref{le-linearalg}. Since 
\begin{eqnarray*}
\perp\big(\text{Sup}_n(\mathcal{H}),\mathbb{F}(\Delta\mathcal{H})_n\big)&=&  E_n,\\
\perp\big(\text{Sup}_{n-1}(\mathcal{H}),\mathbb{F}(\Delta\mathcal{H})_{n-1}\big)&=&  E_{n-1},
\end{eqnarray*}
 we have (\ref{eq-a.4}). Moreover,  if $\partial_n (E_n)\subseteq  E_{n-1}$,  then $\partial_n^*\mid _{\text{Sup}_{n-1}(\mathcal{H})}=(\partial_n\mid _{\text{Sup}_n(\mathcal{H})})^*$,  and the equality in (\ref{eq-a.4}) holds . 
\end{proof}

The next corollary follows from  Proposition~\ref{le-a.1} directly. 

\begin{corollary}
Let $\mathcal{H}$ be a hypergraph and $n\geq 0$. 
\begin{enumerate}[(a).]
\item
If $\partial_i(A_i\oplus B_i\oplus E_i)\subseteq A_{i-1}\oplus B_{i-1}\oplus E_{i-1}$ for $i=n+1$ and $n$, then  
\begin{eqnarray*}
L^{\Delta\mathcal{H}}_n\mid _{\text{Inf}_n(\mathcal{H})}=L^{\text{Inf}_*(\mathcal{H})}_n;
\end{eqnarray*} 
\item
If 
$\partial_i( E_i)\subseteq E_{i-1}$ for $i=n+1$ and $n$, then 
\begin{eqnarray*}
L^{\Delta\mathcal{H}}_n\mid _{\text{Sup}_n(\mathcal{H})}=L^{\text{Sup}_*(\mathcal{H})}_n.
\end{eqnarray*} 
\end{enumerate}
\qed
\end{corollary}

The next theorem characterizes further properties about the kernels of the Laplacians using the embedded homology of hypergraphs.
The proof follows by using Theorem~\ref{pr.a.1} and Proposition~\ref{le-a.1}. 

\begin{theorem}[Hodge Isomorphism for Hypergraphs: Part II]\label{th-a.1}
Let $\mathcal{H}$ be a hypergraph and   $n\geq 0$.   Then both $\text{Ker}L^{\Delta\mathcal{H}}_n\cap \text{Inf}_n(\mathcal{H})$ and $\text{Ker}L^{\Delta\mathcal{H}}_n\cap \text{Sup}_n(\mathcal{H})$ are   subspaces of $H_n(\mathcal{H};\mathbb{F})$.  Moreover, {if}  $\partial_n(A_n\oplus B_n\oplus E_n)\subseteq A_{n-1}\oplus B_{n-1}\oplus E_{n-1}$, then 
\begin{eqnarray*}
\text{Ker}L^{\Delta\mathcal{H}}_n\cap \text{Inf}_n(\mathcal{H})\cong H_n(\mathcal{H};\mathbb{F}). 
\end{eqnarray*}
And {if} 
$\partial_n( E_n)\subseteq E_{n-1}$,  then 
\begin{eqnarray*}
\text{Ker}L^{\Delta\mathcal{H}}_n\cap \text{Sup}_n(\mathcal{H})\cong H_n(\mathcal{H};\mathbb{F}). 
\end{eqnarray*}
\end{theorem}

\begin{proof}
By (\ref{eq-a.2}), 
\begin{eqnarray}
\text{Ker}L^{\Delta\mathcal{H}}_n\cap \text{Inf}_n(\mathcal{H})&=& \text{Ker}\partial_n   \cap  \text{Ker} \partial_{n+1}^* \cap \text{Inf}_n(\mathcal{H}) \nonumber\\
&=&\big(\text{Ker}\partial_n \cap \text{Inf}_n(\mathcal{H})\big)\cap \big(\text{Ker} \partial_{n+1}^* \cap \text{Inf}_n(\mathcal{H})\big)\nonumber\\
&=&\text{Ker}(\partial_n\mid_{\text{Inf}_*(\mathcal{H})})\cap \text{Ker}(\partial^*_{n+1}\mid_{\text{Inf}_*(\mathcal{H})})\label{eq-aa.9}
\end{eqnarray}
and
\begin{eqnarray}\label{eq-aa.8}
\text{Ker}L^{\Delta\mathcal{H}}_n\cap \text{Sup}_n(\mathcal{H})=\text{Ker}(\partial_n\mid_{\text{Sup}_*(\mathcal{H})})\cap \text{Ker}(\partial^*_{n+1}\mid_{\text{Sup}_*(\mathcal{H})}).
\end{eqnarray}
By  Proposition~\ref{le-a.1} and (\ref{eq-aa.9}),  
\begin{eqnarray}\label{eq-bb.9}
\text{Ker}L^{\Delta\mathcal{H}}_n\cap \text{Inf}_n(\mathcal{H})\subseteq \text{Ker}(\partial_n\mid_{\text{Inf}_*(\mathcal{H})})\cap \text{Ker}(\partial_{n}\mid_{\text{Inf}_*(\mathcal{H})})^*. 
\end{eqnarray}
The equality in (\ref{eq-bb.9}) holds if $\partial_n(A_n\oplus B_n\oplus E_n)\subseteq A_{n-1}\oplus B_{n-1}\oplus E_{n-1}$.  
By Theorem~\ref{pr.a.1} and (\ref{eq-bb.9}), $\text{Ker}L^{\Delta\mathcal{H}}_n\cap \text{Inf}_n(\mathcal{H})$ is a subspace of $H_n(\mathcal{H};\mathbb{F})$. And if   $\partial_n(A_n\oplus B_n\oplus E_n)\subseteq A_{n-1}\oplus B_{n-1}\oplus E_{n-1}$, then $\text{Ker}L^{\Delta\mathcal{H}}_n\cap \text{Inf}_n(\mathcal{H})$ is  isomorphic to $H_n(\mathcal{H};\mathbb{F})$. 

By  Proposition~\ref{le-a.1} and (\ref{eq-aa.8}), 
\begin{eqnarray}\label{eq-bb.8}
\text{Ker}L^{\Delta\mathcal{H}}_n\cap \text{Sup}_n(\mathcal{H})\subseteq \text{Ker}(\partial_n\mid_{\text{Sup}_*(\mathcal{H})})\cap \text{Ker}(\partial_{n}\mid_{\text{Sup}_*(\mathcal{H})})^*. 
\end{eqnarray}
The equality in (\ref{eq-bb.8}) holds if $\partial_n( E_n)\subseteq  E_{n-1}$.  
 By Theorem~\ref{pr.a.1} and (\ref{eq-bb.8}),  $\text{Ker}L^{\Delta\mathcal{H}}_n\cap \text{Sup}_n(\mathcal{H})$ is a subspace of $H_n(\mathcal{H};\mathbb{F})$. And if   $\partial_n( E_n)\subseteq E_{n-1}$, then $\text{Ker}L^{\Delta\mathcal{H}}_n\cap \text{Sup}_n(\mathcal{H})$ is  isomorphic to $H_n(\mathcal{H};\mathbb{F})$.
\end{proof}

\section{Hodge Decompositions for Hypergraphs}\label{s-a.3}

In this section, we prove some Hodge decompositions for hypergraphs in Theorem~\ref{th-3.19}.  

\subsection{Orthogonal Decompositions for Homologies of Hypergraphs}

 In this subsection,  we prove some orthogonal  decompositions for the embedded homology of hypergraphs and the homology of associated simplicial complexes, in Theorem~\ref{th-decomp1}. 
 
\smallskip

Let $n\geq 0$.  By Lemma~\ref{le-linearalg},
\begin{eqnarray}\label{eq-bc-7}
\text{Ker}(\partial_{n+1}\mid_{\text{Sup}_*(\mathcal{H})})^*\cap \text{Inf}_n(\mathcal{H})\subseteq\text{Ker}(\partial_{n+1}\mid_{\text{Inf}_*(\mathcal{H})})^*.  
\end{eqnarray}
By Theorem~\ref{pr.a.1} and  (\ref{eq-bc-7}), 
\begin{eqnarray}\nonumber
\text{Ker}(L_n^{\text{Inf}_*(\mathcal{H})})&=&\text{Ker}\partial_n \cap \text{Ker}(\partial_{n+1}\mid_{\text{Inf}_*(\mathcal{H})})^*\\
\nonumber
&\supseteq &\text{Ker}\partial_n \cap \text{Ker}(\partial_{n+1}\mid_{\text{Sup}_*(\mathcal{H})})^*\cap\text{Inf}_n(\mathcal{H})\\
&=&\text{Ker}(L_n^{\text{Sup}_*(\mathcal{H})}) \cap \text{Inf}_n(\mathcal{H}).  
\label{eq-bc-8}
\end{eqnarray}
By Theorem~\ref{pr.a.1}, Theorem~\ref{th-a.1} and (\ref{eq-bc-8}), we have a  diagram of vector spaces and linear maps
\begin{eqnarray}
\xymatrix{
 H_n(\mathcal{H};\mathbb{F})\ar[r]^{\cong} \ar [dd]^{\cong}&\text{Ker}(L_n^{\text{Sup}_*(\mathcal{H})})&\text{Ker}(L_n^{\Delta\mathcal{H}}) \\
 \\
 \text{Ker}(L_n^{\text{Inf}_*(\mathcal{H})})&\text{Ker}(L_n^{\text{Sup}_*(\mathcal{H})})\cap \text{Inf}_n(\mathcal{H})\ar[l]_{i_2}\ar[uu]^{i_3}&\text{Ker}(L_n^{\Delta\mathcal{H}})\cap\text{Sup}_n(\mathcal{H})  \ar[luu]_{i_6}  \ar[uu]^{i_7} \\
 \\
 \text{Ker}(L_n^{\Delta\mathcal{H}})\cap\text{Inf}_n(\mathcal{H})\ar[uu]^{i_1}\ar[r]^{i_4}& \text{Ker}(L_n^{\Delta\mathcal{H}})\cap \mathbb{F}(\mathcal{H})_n\ar[ruu]^{i_5}.& 
 }
 \label{diag-1}
\end{eqnarray}
Here $i_1$ and $i_6$ are the inclusions given by Theorem~\ref{th-a.1}, $i_3$, $i_4$, $i_5$ and $i_7$ are the canonical inclusions,  and $i_2$ is the inclusion given by (\ref{eq-bc-8}).   The next proposition follows from the  above diagram (\ref{diag-1}). 

\begin{proposition}\label{pr-decomp0}
Let $\mathcal{H}$ be a hypergraph and $n\geq 0$. Then we have 
\begin{enumerate}[(a).]
\item
the orthogonal decomposition of the embedded homology into four summands
\begin{eqnarray*}
H_n(\mathcal{H};\mathbb{F})&\cong & \big( \text{Ker}(L_n^{\Delta\mathcal{H}})\cap\text{Inf}_n(\mathcal{H}) \big)\\
&&\oplus  \big(  \text{Ker}(L_n^{\Delta\mathcal{H}})\cap \perp (\text{Inf}_n(\mathcal{H}), \mathbb{F}(\mathcal{H})_n )\big)\\
&&\oplus \big(  \text{Ker}(L_n^{\Delta\mathcal{H}})\cap \perp (\mathbb{F}(\mathcal{H})_n,\text{Sup}_n(\mathcal{H}))\big)\\
&&\oplus\perp\big( \text{Ker}(L_n^{\Delta\mathcal{H}})\cap\text{Sup}_n(\mathcal{H}), \text{Ker}(L_n^{\text{Sup}_*(\mathcal{H})})\big); 
\end{eqnarray*}
\item
 the orthogonal decomposition of the homology of $\Delta\mathcal{H}$ into four summands
\begin{eqnarray*}
H_n(\Delta\mathcal{H};\mathbb{F})&\cong & \big( \text{Ker}(L_n^{\Delta\mathcal{H}})\cap\text{Inf}_n(\mathcal{H}) \big)\\
&&\oplus  \big(  \text{Ker}(L_n^{\Delta\mathcal{H}})\cap \perp (\text{Inf}_n(\mathcal{H}), \mathbb{F}(\mathcal{H})_n )\big)\\
&&\oplus \big( \text{Ker}(L_n^{\Delta\mathcal{H}})\cap \perp (\mathbb{F}(\mathcal{H})_n,\text{Sup}_n(\mathcal{H}))\big)\\
&&\oplus\perp\big( \text{Ker}(L_n^{\Delta\mathcal{H}})\cap\text{Sup}_n(\mathcal{H}), \text{Ker}(L_n^{\Delta\mathcal{H}})\big). 
\end{eqnarray*}
\end{enumerate}
\end{proposition}
\begin{proof}(a). 
By the map $i_4$, 
\begin{eqnarray}\nonumber
\text{Ker}(L_n^{\Delta\mathcal{H}})\cap \mathbb{F}(\mathcal{H})_n&=& \big(\text{Ker}(L_n^{\Delta\mathcal{H}})\cap\text{Inf}_n(\mathcal{H}) \big) \\
&&\oplus  \big(  \text{Ker}(L_n^{\Delta\mathcal{H}})\cap \perp (\text{Inf}_n(\mathcal{H}), \mathbb{F}(\mathcal{H})_n )\big). \label{eq-x-1}
\end{eqnarray}
By the map $i_5$,
\begin{eqnarray}\nonumber
\text{Ker}(L_n^{\Delta\mathcal{H}})\cap\text{Sup}_*(\mathcal{H}) &=& \big(\text{Ker}(L_n^{\Delta\mathcal{H}})\cap \mathbb{F}(\mathcal{H})_n \big) \\
&&\oplus \big(  \text{Ker}(L_n^{\Delta\mathcal{H}})\cap \perp (\mathbb{F}(\mathcal{H})_n,\text{Sup}_n(\mathcal{H}))\big). \label{eq-x-2}
\end{eqnarray}
By the map $i_6$, 
\begin{eqnarray}\nonumber
\text{Ker}(L_n^{\text{Sup}_*(\mathcal{H})}) &=& \big(\text{Ker}(L_n^{\Delta\mathcal{H}})\cap\text{Sup}_*(\mathcal{H}) 
 \big) \\
&&\oplus\perp\big(\text{Ker}(L_n^{\Delta\mathcal{H}})\cap\text{Sup}_n(\mathcal{H}), \text{Ker}(L_n^{\text{Sup}_*(\mathcal{H})})\big). \label{eq-x-3}
\end{eqnarray}
Since $H_n(\mathcal{H};\mathbb{F})=H_n(\text{Sup}_*(\mathcal{H}))\cong \text{Ker}(L_n^{\text{Sup}_*(\mathcal{H})})$, the decomposition follows from (\ref{eq-x-1}), (\ref{eq-x-2}) and (\ref{eq-x-3}).  

(b). Following from the maps $i_4$, $i_5$ and $i_7$, the proof of (b) is similar with the proof of (a). 
\end{proof}

We study the summands of the orthogonal decompositions in Proposition~\ref{pr-decomp0}.

 \begin{enumerate}[(I). ]
 \item
 Let $\alpha\in \mathbb{F}(\Delta\mathcal{H})_n$. Then 
 \begin{eqnarray*}
 \alpha\in \text{Ker}(L_n^{\Delta\mathcal{H}})\cap\text{Inf}_n(\mathcal{H})
 \end{eqnarray*}
  if and only if all the following four conditions are satisfied:
 \begin{enumerate}[(i)'. ]
\item
$\alpha\in \text{Ker}\partial_n$,
\item
$\alpha\in \text{Ker}(\partial_{n+1}^*)$,
\item
$\alpha\in\mathbb{F}(\mathcal{H})_n$,
\item 
$\alpha\in \partial_n^{-1}\mathbb{F}(\mathcal{H})_{n-1}$.
 \end{enumerate}
 Hence the space  
 \begin{eqnarray}\label{eq-3.4.097}
 \text{Ker}(L_n^{\Delta\mathcal{H}})\cap\text{Inf}_n(\mathcal{H})
 \end{eqnarray}
  is the collection of all $\alpha \in \mathbb{F}(\mathcal{H})_n$ such that both of the following two conditions are satisfied: 
 \begin{enumerate}[(i).]
 \item
 $\partial_n\alpha=0$,
 \item
 for any $\beta\in \mathbb{F}(\Delta\mathcal{H})_{n+1}$, $\langle \partial_{n+1}\beta,\alpha\rangle =0$. 
 \end{enumerate}
 
 \item
  Let $\alpha\in \mathbb{F}(\Delta\mathcal{H})_n$. Then 
  \begin{eqnarray*}
  \alpha\in  \text{Ker}(L_n^{\Delta\mathcal{H}})\cap \perp \big(\text{Inf}_n(\mathcal{H}), \mathbb{F}(\mathcal{H})_n \big)
  \end{eqnarray*}
   if and only if all the following four conditions are satisfied:
 \begin{enumerate}[(i)'. ]
\item
$\alpha\in \text{Ker}\partial_n$,
\item
$\alpha\in \text{Ker}(\partial_{n+1}^*)$,
\item
$\alpha\in\mathbb{F}(\mathcal{H})_n$,
\item 
for any $\gamma\in \text{Inf}_n\mathcal{H}$, $\langle \gamma,\alpha\rangle =0$. 
 \end{enumerate}
 Hence the space 
 \begin{eqnarray}\label{eq-3.4.091}
  \text{Ker}(L_n^{\Delta\mathcal{H}})\cap \perp \big(\text{Inf}_n(\mathcal{H}), \mathbb{F}(\mathcal{H})_n \big)
  \end{eqnarray}
   is the collection of all  $\alpha \in \mathbb{F}(\mathcal{H})_n$ such that all of the following three conditions are satisfied: 
 \begin{enumerate}[(i).]
 \item
 $\partial_n\alpha=0$,
 \item
 for any $\beta\in \mathbb{F}(\Delta\mathcal{H})_{n+1}$, $\langle \partial_{n+1}\beta,\alpha\rangle =0$,
 \item
 for any $\gamma\in \mathbb{F}(\mathcal{H})_n$, if $\partial_n\gamma\in \mathbb{F}(\mathcal{H})_{n-1}$, then $\langle\gamma,\alpha\rangle=0$. 
 \end{enumerate}

 \item
 Let $\alpha\in \mathbb{F}(\Delta\mathcal{H})_n$. Then 
 \begin{eqnarray*}
 \alpha\in   \text{Ker}(L_n^{\Delta\mathcal{H}})\cap \perp \big(\mathbb{F}(\mathcal{H})_n,\text{Sup}_n(\mathcal{H})\big)
 \end{eqnarray*}
  if and only if all the following four conditions are satisfied:
 \begin{enumerate}[(i)'. ]
\item
$\alpha\in \text{Ker}\partial_n$,
\item
$\alpha\in \text{Ker}(\partial_{n+1}^*)$,
\item
$\alpha=\theta_n+\partial_{n+1}\theta_{n+1}$ for some $\theta_n\in\mathbb{F}(\mathcal{H})_n$ and $\theta_{n+1}\in \mathbb{F}(\mathcal{H})_{n+1}$,
\item 
for any $\gamma\in \mathbb{F}(\mathcal{H})_n$, $\langle \gamma,\alpha\rangle =0$. 
 \end{enumerate}
 Hence the space 
 \begin{eqnarray}\label{eq-3.4.092}
 \text{Ker}(L_n^{\Delta\mathcal{H}})\cap \perp \big(\mathbb{F}(\mathcal{H})_n,\text{Sup}_n(\mathcal{H})\big)
 \end{eqnarray}
  is the collection of all  $\theta_n+\partial_{n+1}\theta_{n+1}$, where $\theta_n\in \mathbb{F}(\mathcal{H})_n$ and $\theta_{n+1}\in\mathbb{F}(\mathcal{H})_{n+1}$,  such that all of the following three conditions are satisfied: 
 \begin{enumerate}[(i).]
 \item
 $\partial_n\theta_n=0$,
 \item
 for any $\beta\in \mathbb{F}(\Delta\mathcal{H})_{n+1}$, $\langle \partial_{n+1}\beta, \theta_n+\partial_{n+1}\theta_{n+1} \rangle =0$,
 \item
 for any $\gamma\in \mathbb{F}(\mathcal{H})_n$,   $\langle\gamma, \theta_n+\partial_{n+1}\theta_{n+1}\rangle=0$. 
 \end{enumerate}

 \item
  Let $\alpha\in \mathbb{F}(\Delta\mathcal{H})_n$. Then 
  \begin{eqnarray*}
  \alpha\in \perp \big(\text{Ker}(L_n^{\Delta\mathcal{H}})\cap\text{Sup}_n(\mathcal{H}), \text{Ker}(L_n^{\text{Sup}_*(\mathcal{H})}) \big)
  \end{eqnarray*}
   if and only if all the following three conditions are satisfied:
 \begin{enumerate}[(i)'. ]
\item
$\alpha\in \text{Ker}\partial_n \cap \text{Sup}_n(\mathcal{H})$,
\item
$\alpha\in \text{Ker}(\partial_{n+1}\mid_{\text{Sup}_*(\mathcal{H})})^*$,
\item
for any $\alpha'$ satisfying (i) and (ii), if $\alpha'\in \text{Ker}(\partial_n^*)$,  then $\langle\alpha',\alpha\rangle=0$. 
 \end{enumerate}
 Hence the space 
 \begin{eqnarray}\label{eq-3.4.098}
 \perp \big(\text{Ker}(L_n^{\Delta\mathcal{H}})\cap\text{Sup}_n(\mathcal{H}), \text{Ker}(L_n^{\text{Sup}_*(\mathcal{H})})\big )
 \end{eqnarray}
  is the collection of all  $\theta_n+\partial_{n+1}\theta_{n+1}$, where $\theta_n\in \mathbb{F}(\mathcal{H})_n$ and $\theta_{n+1}\in\mathbb{F}(\mathcal{H})_{n+1}$,  such that all of the following three conditions are satisfied: 
 \begin{enumerate}[(i).]
 \item
 $\partial_n\theta_n=0$,
 \item
 for any $\gamma\in \mathbb{F}(\mathcal{H})_{n}$, $\langle \partial_{n}\gamma, \theta_n+\partial_{n+1}\theta_{n+1} \rangle =0$,
 \item
 for any $\theta'_n\in \mathbb{F}(\mathcal{H})_n$ and $\theta'_{n+1}\in\mathbb{F}(\mathcal{H})_{n+1}$ satisfying (i) and (ii),  if for any $\beta\in \mathbb{F}(\Delta\mathcal{H})_{n+1}$, $\langle \partial_{n+1}\beta,\theta'_n+\partial_{n+1}\theta'_{n+1}\rangle =0$,  then $\langle \theta'_n+\partial_{n+1}\theta'_{n+1}, \theta_n+\partial_{n+1}\theta_{n+1} \rangle=0$. 
 \end{enumerate}
 
 \item
  Let $\alpha\in \mathbb{F}(\Delta\mathcal{H})_n$. Then 
  \begin{eqnarray*}
  \alpha\in\perp \big(\text{Ker}(L_n^{\Delta\mathcal{H}})\cap\text{Sup}_n(\mathcal{H}), \text{Ker}(L_n^{\Delta\mathcal{H}}) \big)
  \end{eqnarray*}
   if and only if all the following three conditions are satisfied:
 \begin{enumerate}[(i)'. ]
\item
$\alpha\in \text{Ker}\partial_n$,
\item
$\alpha\in \text{Ker}(\partial_{n+1}^*)$,
\item
for any $\alpha'$ satisfying (i) and (ii), if $\alpha'\in \text{Sup}_n(\mathcal{H})$,  then $\langle\alpha',\alpha\rangle=0$. 
 \end{enumerate}
 Hence the space 
 \begin{eqnarray}\label{eq-3.4.876}
 \perp\big (\text{Ker}(L_n^{\Delta\mathcal{H}})\cap\text{Sup}_n(\mathcal{H}), \text{Ker}(L_n^{\Delta\mathcal{H}})\big )
 \end{eqnarray}
  is the collection of all  $\alpha\in \mathbb{F}(\Delta\mathcal{H})_n$  such that all of the following three conditions are satisfied: 
 \begin{enumerate}[(i).]
 \item
 $\partial_n\alpha=0$,
 \item
 for any $\beta\in \mathbb{F}(\Delta\mathcal{H})_{n+1}$, $\langle \partial_{n+1}\beta,\alpha\rangle =0$,
 \item
 for any for any $\alpha'$ satisfying (i) and (ii), if  $\alpha' = \theta_n +\partial_{n+1}\theta_{n+1}$ for some $\theta _n\in \mathbb{F}(\mathcal{H})_n$ and $\theta'_{n+1}\in\mathbb{F}(\mathcal{H})_{n+1}$,    then $\langle \alpha',\alpha \rangle=0$. 
 \end{enumerate}
 \end{enumerate}

It  follows from (I)-(i), (I)-(ii), (II)-(i), (II)-(ii)  and  (II)-(iii) that  (\ref{eq-3.4.091}) 
 is a subspace of (\ref{eq-3.4.097}). 
   Since the two spaces are orthogonal,   the space (\ref{eq-3.4.091}) is zero.  
The next proposition follows.
\begin{proposition}\label{pr3.1.a.b.1}
In the  diagram (\ref{diag-1}), the map $i_4$ is an isomorphism. 
\qed
\end{proposition}

By (III)-(iii),  we have the following commutative diagram
\begin{eqnarray*}
\xymatrix{
\mathbb{F}(\mathcal{H})_{n+1} \ar[rrdd]_{p\circ \partial_{n+1}}\ar[rr]^{\partial_{n+1}} && \text{Sup}_n(\mathcal{H}) \ar[dd]^{\text{orthogonal proj.}}_p
\\
\\
&&\mathbb{F}(\mathcal{H})_n
}
\end{eqnarray*}
and 
\begin{eqnarray}\label{eq-oct-1}
\theta_n=-p\circ\partial_{n+1} (\theta_{n+1}).
\end{eqnarray}
  Let $\beta=\theta_{n+1}$ in (III)-(ii).  With the help of (\ref{eq-oct-1}),  we have
  \begin{eqnarray*}
  \langle \partial_{n+1}\theta_{n+1},\partial_{n+1}\theta_{n+1}\rangle = \langle \partial_{n+1}\theta_{n+1}, p(\partial_{n+1} \theta_{n+1})\rangle. 
  \end{eqnarray*}
  Therefore,  $p$ is the identity map on $\partial_{n+1}\theta_{n+1}$, and 
  $\partial_{n+1}\theta_{n+1}\in \mathbb{F}(\mathcal{H})_n$.  Consequently,  
  \begin{eqnarray*}
  \theta_n+ \partial_{n+1}\theta_{n+1}=0. 
  \end{eqnarray*}
  Hence  the space (\ref{eq-3.4.092}) is zero. 
   The next proposition follows.
\begin{proposition}\label{pr3.1.a.b.2}
In the diagram (\ref{diag-1}), the map $i_5$ is an isomorphism. 
\qed
\end{proposition}

The next theorem follows from Proposition~\ref{pr-decomp0}, Proposition~\ref{pr3.1.a.b.1} and Proposition~\ref{pr3.1.a.b.2}.

\begin{theorem} 
\label{th-decomp1}
Let $\mathcal{H}$ be a hypergraph and $n\geq 0$. Then we have 
\begin{enumerate}[(a).]
\item
the orthogonal decomposition of the embedded homology into two summands
\begin{eqnarray*}
H_n(\mathcal{H};\mathbb{F})&\cong & \big( \text{Ker}(L_n^{\Delta\mathcal{H}})\cap\text{Inf}_n(\mathcal{H}) \big)\\
&&\oplus\perp\big( \text{Ker}(L_n^{\Delta\mathcal{H}})\cap\text{Sup}_n(\mathcal{H}), \text{Ker}(L_n^{\text{Sup}_*(\mathcal{H})})\big); 
\end{eqnarray*}
\item
 the orthogonal decomposition of the homology of $\Delta\mathcal{H}$ into two summands
\begin{eqnarray*}
H_n(\Delta\mathcal{H};\mathbb{F})&\cong & \big( \text{Ker}(L_n^{\Delta\mathcal{H}})\cap\text{Inf}_n(\mathcal{H}) \big)\\
&&\oplus\perp\big( \text{Ker}(L_n^{\Delta\mathcal{H}})\cap\text{Sup}_n(\mathcal{H}), \text{Ker}(L_n^{\Delta\mathcal{H}})\big). 
\end{eqnarray*}
\end{enumerate}
\qed
\end{theorem}

\begin{remark}
In particular, suppose $\mathcal{H}$ is a simplicial complex.  Then the maps  $i_6$ and $i_7$ are  both  isomorphisms. Hence the decompositions in Theorem~\ref{th-decomp1} are trivial for simplicial complexes. 
\end{remark}

\begin{remark}\label{re-81}
By the map $i_1$ in the diagram (\ref{diag-1}),  we obtain a decomposition
 \begin{eqnarray*}
H_n(\mathcal{H};\mathbb{F})&\cong & \big(\text{Ker}(L_n^{\Delta\mathcal{H}})\cap\text{Inf}_n(\mathcal{H}) \big)\\
&&\oplus\perp\big( \text{Ker}(L_n^{\Delta\mathcal{H}})\cap\text{Inf}_n(\mathcal{H}), \text{Ker}(L_n^{\text{Inf}_*(\mathcal{H})})\big),   
\end{eqnarray*} 
which is equivalent to Theorem~\ref{th-decomp1}~(a).  
\end{remark}
\begin{remark}
The summands of the decompositions in Theorem~\ref{th-decomp1} are characterized by (I), (IV) and (V). 
\end{remark}

By the Hodge decomposition of simplicial complexes (cf. \cite{duv,eck,adv1}), we have  
an orthogonal decomposition  
\begin{eqnarray}\label{eq2.8.8.8}
\mathbb{F}(\Delta\mathcal{H})_n&=& \text{Ker}(L_n^{\Delta\mathcal{H}})\oplus \partial_{n+1}( {\mathbb{F}(\Delta\mathcal{H})_{n+1}}) 
  \oplus\partial^*_{n} (\mathbb{F}(\Delta\mathcal{H})_{n-1})\nonumber\\
  &\cong&H_n(\Delta\mathcal{H};\mathbb{F})\oplus \partial_{n+1}( {\mathbb{F}(\Delta\mathcal{H})_{n+1}}) 
  \oplus\partial^*_{n} (\mathbb{F}(\Delta\mathcal{H})_{n-1}). 
  \end{eqnarray}
In general, we have the  Hodge decomposition of chain complexes. 

\begin{lemma}[Hodge Decomposition  of Chain Complexes]\label{lemma2.99999}
  Let $C_*$ be a graded Euclidean space  over $\mathbb{F}$ with maps $d_{n+1}: C_{n+1}\longrightarrow C_{n}$ such that $d_{n+1}d_n=0$   for each $n\geq 0$.  Let $L_n=d_{n+1}d_{n+1}^*+ d_n^* d_n$. Then we have the orthogonal decomposition
\begin{eqnarray}\label{eq2.8.8.8.8}
C_n= \text{Ker}(L_n)\oplus d_{n+1}  C_{n+1} \oplus d_{n}^* C_{n-1}. 
\end{eqnarray}
\end{lemma}

\begin{proof}
Lemma~\ref{lemma2.99999} is an analogue of \cite[Theorem~4.18]{morita} for chain complexes.  With minor modifications, the proof of \cite[Theorem~4.18]{morita} applies.  
\end{proof}

The next corollary follows from Theorem~\ref{th-decomp1}~(b) and  (\ref{eq2.8.8.8}). 

\begin{corollary}\label{co-3.3.x}
Let $\mathcal{H}$ be a hypergraph and $n\geq 0$. Then we have the orthogonal decomposition of the vector space spanned by the $n$-simplices of $\Delta\mathcal{H}$ into four summands
\begin{eqnarray*}
\mathbb{F}(\Delta\mathcal{H})_n&=& \big(  \text{Ker}(L_n^{\Delta\mathcal{H}})\cap\text{Inf}_n(\mathcal{H}) \big)\\
&&\oplus\perp\big(  \text{Ker}(L_n^{\Delta\mathcal{H}})\cap\text{Sup}_n(\mathcal{H}),  \text{Ker}(L_n^{\Delta\mathcal{H}})\big)\\
&&\oplus   \partial_{n+1}\big( {\mathbb{F}(\Delta\mathcal{H})_{n+1}}\big) 
  \oplus\partial^*_{n} \big(\mathbb{F}(\Delta\mathcal{H})_{n-1}\big). 
\end{eqnarray*}
\qed
\end{corollary}

The next corollary follows from Theorem~\ref{th-decomp1}~(a) and Lemma~\ref{lemma2.99999}. 

\begin{corollary}\label{co-decomp2}
Let $\mathcal{H}$ be a hypergraph and $n\geq 0$. Then we have the orthogonal decomposition of the $n$-dimensional space of the supremum chain complex into four summands
\begin{eqnarray*}
\text{Sup}_n(\mathcal{H})&=& \big(  \text{Ker}(L_n^{\Delta\mathcal{H}})\cap\text{Inf}_n(\mathcal{H}) \big)\\
&&\oplus\perp\big(  \text{Ker}(L_n^{\Delta\mathcal{H}})\cap\text{Sup}_n(\mathcal{H}), \text{Ker}(L_n^{\text{Sup}_*(\mathcal{H})})\big)\\
&&\oplus  \partial_{n+1} \text{Sup}_{n+1}(\mathcal{H}) \oplus (\partial_{n}\mid _{\text{Sup}_*(\mathcal{H})})^* \text{Sup}_{n-1}(\mathcal{H}). 
\end{eqnarray*}
\qed
\end{corollary}

\subsection{Some Examples}

 In this subsection, we   give some examples  of hypergraphs such that the  decompositions in Theorem~\ref{th-decomp1} are non-trivial.  
 
 \smallskip
 
 The next example shows that the decomposition of Theorem~\ref{th-decomp1}~(a)  is non-trivial. 

\begin{example}
Let $n\geq 3$. Let $\Delta[n]$ be the simplicial complex consisting of the standard $n$-simplex $\sigma^n$ (with $n+1$ vertices) together with all its faces.   We consider the hypergraphs
\begin{eqnarray*}
\mathcal{H}^1&=& \text{Sk}^1(\Delta[n]),\\
\mathcal{H}^2&=& \{\sigma^n\} \sqcup\text{Sk}^1(\Delta[n])
\end{eqnarray*}
and
\begin{eqnarray*}
 \mathcal{H}=\mathcal{H}^1\sqcup \mathcal{H}^2.
 \end{eqnarray*}
Here $\text{Sk}^1$ denotes the $1$-skeleton and $\sqcup$ denotes the disjoint union.  Then
\begin{eqnarray*}
H_1(\mathcal{H}^1;\mathbb{F})&=&H_1(\mathcal{H}^2;\mathbb{F})\\
&=&H_1(\text{Sk}^1(\Delta[n]);\mathbb{F})\\
&=& \mathbb{F}^{\oplus {{n}\choose{2}}}. 
\end{eqnarray*}
Hence
\begin{eqnarray*}
&&\text{Ker}(L_n^{\text{Sup}_*(\mathcal{H})})\cong H_1(\mathcal{H};\mathbb{F})\\
&=& H_1(\mathcal{H}^1;\mathbb{F})\oplus H_1(\mathcal{H}^2;\mathbb{F})=\mathbb{F}^{\oplus 2{{n}\choose{2}}}. 
\end{eqnarray*}
On the other hand, $\Delta\mathcal{H}^1=\Delta[n]$,  $\Delta\mathcal{H}^2=\text{Sk}^1(\Delta[n])$ and $\Delta\mathcal{H}=\Delta\mathcal{H}^1\sqcup \Delta\mathcal{H}^2$. Hence
\begin{eqnarray*}
 \text{Ker}(L_1^{\Delta\mathcal{H}})\cong H_1(\Delta\mathcal{H};\mathbb{F})= \mathbb{F}^{\oplus {{n}\choose{2}}}.  
\end{eqnarray*}
Moreover,
\begin{eqnarray*}
\text{Inf}_1(\mathcal{H})&=&\text{Sup}_1(\mathcal{H})\\
&=&\mathbb{F}(\mathcal{H}^1)_1\oplus \mathbb{F}(\mathcal{H}^2)_1\\
&=&\big(\mathbb{F}\big(\text{Sk}^1(\Delta[n])\big)_1\big)^{\oplus 2}. 
\end{eqnarray*}
Thus the two summands of the decomposition of $H_1(\mathcal{H};\mathbb{F})$ in Theorem~\ref{th-decomp1}~(a) are  
\begin{eqnarray*}
 \text{Ker}(L_1^{\Delta\mathcal{H}})\cap \text{Inf}_1(\mathcal{H})&=&\mathbb{F}^{\oplus {{n}\choose{2}}},\\
\perp\big( \text{Ker}(L_1^{\Delta\mathcal{H}})\cap \text{Sup}_1(\mathcal{H}),\text{Ker}(L_1^{\text{Sup}_*(\mathcal{H})})\big)&=&\mathbb{F}^{\oplus {{n}\choose{2}}}. 
\end{eqnarray*} 
The decomposition of Theorem~\ref{th-decomp1}~(a)  is $\mathbb{F}^{\oplus 2{{n}\choose{2}}}=\mathbb{F}^{\oplus  {{n}\choose{2}}}\oplus \mathbb{F}^{\oplus  {{n}\choose{2}}}$. 
\end{example}

 The next example shows that the decomposition of Theorem~\ref{th-decomp1}~(b)  is non-trivial.

\begin{example}\label{ex-3.4.2}
We consider the hypergraphs
\begin{eqnarray*}
\mathcal{H}^1=\big\{\{v_0,v_1,v_3\}, \{v_1,v_2,v_4\},\{v_3,v_4,v_5\}\big\} 
\end{eqnarray*} 
and $\mathcal{H}=\mathcal{H}^1\sqcup \Delta\mathcal{H}^1$.    Then $\mathcal{H}$ is the hypergraph drawn in  Figure~\ref{fig1}. 
\begin{figure}
 \begin{center}
\begin{tikzpicture}
\coordinate [label=left:$v_0$]    (A) at (2/2,0); 
 \coordinate [label=right:$v_1$]   (B) at (5/2,0); 
 \coordinate  [label=right:$v_2$]   (C) at (8/2,0); 
\coordinate  [label=right:$v_3$]   (D) at (3.5/2,2/2); 
\coordinate  [label=right:$v_4$]   (E) at (6.5/2,2/2); 
\coordinate  [label=right:$v_5$]   (F) at (5/2,4/2); 

 \coordinate[label=left:$\mathcal{H}^1$:] (G) at (0.5/2,2/2);
 \draw [dotted] (A) -- (B);
 \draw [dotted] (B) -- (C);
  \draw [dotted] (D) -- (A);
\draw [dotted] (D) -- (B);
\draw [dotted] (E) -- (C);
\draw [dotted] (E) -- (B);
\draw [dotted] (E) -- (F);
\draw [dotted] (D) -- (F);
\draw [dotted] (D) -- (E);

\fill [fill opacity=0.18][gray!100!white] (A) -- (B) -- (D) -- cycle;
\fill [fill opacity=0.18][gray!100!white] (B) -- (C) -- (E) -- cycle;
\fill [fill opacity=0.18][gray!100!white] (D) -- (E) -- (F) -- cycle;

\coordinate [label=left:$u_0$]    (A) at (1+6,0); 
 \coordinate [label=right:$u_1$]   (B) at (2.5+6,0); 
 \coordinate  [label=right:$u_2$]   (C) at (4+6,0); 
\coordinate  [label=right:$u_3$]   (D) at (1.75+6,2/2); 
\coordinate  [label=right:$u_4$]   (E) at (6.5/2+6,2/2); 
\coordinate  [label=right:$u_5$]   (F) at (5/2+6,4/2); 

 \coordinate[label=left:$\Delta{\mathcal{H}}^1$:] (G) at (0.5/2+6,2/2);
 \draw  [line width=1pt](A) -- (B);
 \draw [line width=1pt] (B) -- (C);
  \draw [line width=1pt] (D) -- (A);
\draw [line width=1pt] (D) -- (B);
\draw [line width=1pt] (E) -- (C);
\draw [line width=1pt] (E) -- (B);
\draw [line width=1pt] (E) -- (F);
\draw [line width=1pt] (D) -- (F);
\draw [line width=1pt] (D) -- (E);

\fill [fill opacity=0.18][gray!100!white] (A) -- (B) -- (D) -- cycle;
\fill [fill opacity=0.18][gray!100!white] (B) -- (C) -- (E) -- cycle;
\fill [fill opacity=0.18][gray!100!white] (D) -- (E) -- (F) -- cycle;

\fill (7,0) circle (2.5pt);
\fill (8.5,0) circle (2.5pt);
\fill (10,0) circle (2.5pt);
\fill (7.75,2/2) circle (2.5pt); 
\fill (6.5/2+6,2/2) circle (2.5pt); 
\fill (5/2+6,4/2) circle (2.5pt); 

\end{tikzpicture}

\end{center}
\caption{Example~\ref{ex-3.4.2}.}
\label{fig1}
\end{figure}
Since $\text{Inf}_1(\mathcal{H}^1)=0$ and $\text{Inf}_1(\Delta\mathcal{H}^1)= \mathbb{F}(\Delta\mathcal{H}^1)_1$,  we have
\begin{eqnarray*}
 \text{Ker}(L_1^{\Delta\mathcal{H}})\cap\text{Inf}_1(\mathcal{H})&=&\big(\text{Ker}(L_1^{\Delta\mathcal{H}^1})\cap\text{Inf}_1(\mathcal{H}^1)\big)\\
&&\oplus \big( \text{Ker}(L_1^{\Delta\mathcal{H}^1})\cap\text{Inf}_1(\Delta\mathcal{H}^1)\big)\\
&=&\mathbb{F}.
\end{eqnarray*}
On the other hand,
\begin{eqnarray*}
\text{Ker}(L_1^{\Delta\mathcal{H}})\cong H_1(\Delta\mathcal{H};\mathbb{F})=\mathbb{F}^{\oplus 2}. 
\end{eqnarray*}
Hence the decomposition of Theorem~\ref{th-decomp1}~(b) is $\mathbb{F}^{\oplus 2}=\mathbb{F}\oplus \mathbb{F}$. 
 \end{example}

\subsection{Isomorphisms of The Embedded Homology}

In this subsection, we prove that the maps $i_2$ and  $i_3$ in the diagram (\ref{diag-1}) are isomorphisms. 

\smallskip

  Consider the canonical inclusion  $\iota: \text{Inf}_*(\mathcal{H})\longrightarrow \text{Sup}_*(\mathcal{H})$.  Then $\iota$ is a chain map.  For each $n\geq 0$,  $\iota$ induces an isomorphism $\iota_*: H_n(\text{Inf}_*(\mathcal{H}))\longrightarrow H_n(\text{Sup}_*(\mathcal{H}))$.  Hence we have a commutative diagram

\begin{eqnarray}
\xymatrix{
\text{Ker}\partial_n\cap\text{Inf}_n(\mathcal{H})\ar[rr]^\iota\ar[dd]_{q_1}&& \text{Ker}\partial_n \cap \text{Sup}_n(\mathcal{H})\ar[dd]_{q_2}\\
\\
\big(\text{Ker}\partial_n\cap\text{Inf}_n(\mathcal{H})\big)\big/ \partial_{n+1} \text{Inf}_{n+1}(\mathcal{H})\ar@{=}[dd]&&\big(\text{Ker}\partial_n\cap\text{Sup}_n(\mathcal{H})\big)\big/ \partial_{n+1} \text{Sup}_{n+1}(\mathcal{H})\ar@{=}[dd]\\
\\
H_n(\text{Inf}_*(\mathcal{H}))\ar[rr]^{\iota_*}_{\cong} &&H_n(\text{Sup}_*(\mathcal{H})). 
}
\label{diag-2}
\end{eqnarray}
Here $q_1$ and $q_2$ are the canonical quotient maps. 

\begin{proposition}\label{pr-3.2.1}
Let $\mathcal{H}$ be a hypergraph and $n\geq 0$. Then 
\begin{eqnarray}\label{eq-xyz-3}
\text{Ker}\partial_n \cap \text{Sup}_n(\mathcal{H})=  \text{Ker}\partial_n\cap\text{Inf}_n(\mathcal{H}) +\partial_{n+1} \mathbb{F}(\mathcal{H})_{n+1}. 
\end{eqnarray}  
Moreover, 
\begin{eqnarray}\label{eq-xyz-6}
\text{Ker}\partial_n \cap \perp \big(\text{Inf}_n(\mathcal{H}), \text{Sup}_n(\mathcal{H})\big)=\perp\big(\partial_{n+1} \mathbb{F}(\mathcal{H})_{n+1} \cap \mathbb{F}(\mathcal{H})_n, \partial_{n+1} \mathbb{F}(\mathcal{H})_{n+1}\big).  
\end{eqnarray}
\end{proposition}
\begin{proof}
Let $x\in \text{Ker}\partial_n \cap \text{Sup}_n(\mathcal{H})$.  Since $\iota_*$ is an isomorphism, there exists $y\in  \text{Ker}\partial_n\cap\text{Inf}_n(\mathcal{H})$ such that $\iota_*q_1 y=q_2 x$. That is, 
\begin{eqnarray}\label{eq-xy-11}
\iota_*(y+ \partial_{n+1}(\text{Inf}_{n+1}(\mathcal{H}))= x+ \partial_{n+1}(\text{Sup}_{n+1}(\mathcal{H})). 
\end{eqnarray}
Since $\iota$ is the canonical inclusion, it follows from (\ref{eq-xy-11})  that
\begin{eqnarray*}
x-y\in \partial_{n+1}(\text{Sup}_{n+1}(\mathcal{H})).  
\end{eqnarray*}
Hence $x=y+\partial_{n+1}z$ for some $z\in \text{Sup}_{n+1}(\mathcal{H})$.  Therefore, 
\begin{eqnarray}\label{eq-xyz-1}
\text{Ker}\partial_n \cap \text{Sup}_n(\mathcal{H})\subseteq  \text{Ker}\partial_n\cap\text{Inf}_n(\mathcal{H}) +  \partial_{n+1} \text{Sup}_{n+1}(\mathcal{H}).
\end{eqnarray}
On the other hand,  since
\begin{eqnarray*}
\text{Ker}\partial_n\cap\text{Inf}_n(\mathcal{H})\subseteq \text{Ker}\partial_n\cap\text{Sup}_n(\mathcal{H})
\end{eqnarray*}
and
\begin{eqnarray*}
\partial_{n+1} \text{Sup}_{n+1}(\mathcal{H}) \subseteq  \text{Ker}\partial_n \cap \text{Sup}_n(\mathcal{H}), 
\end{eqnarray*}
we have
\begin{eqnarray}\label{eq-xyz-21}
\text{Ker}\partial_n \cap \text{Sup}_n(\mathcal{H})\supseteq  \text{Ker}\partial_n\cap\text{Inf}_n(\mathcal{H}) +  \partial_{n+1} \text{Sup}_{n+1}(\mathcal{H}).
\end{eqnarray}
By (\ref{eq-xyz-1}) and (\ref{eq-xyz-21}), 
\begin{eqnarray}\label{eq-xyz-22}
\text{Ker}\partial_n \cap \text{Sup}_n(\mathcal{H})=   \text{Ker}\partial_n\cap\text{Inf}_n(\mathcal{H}) +  \partial_{n+1} \text{Sup}_{n+1}(\mathcal{H}).
\end{eqnarray}
Moreover,
\begin{eqnarray}
\partial_{n+1} \text{Sup}_{n+1}(\mathcal{H})&=&\partial_{n+1} \big(\mathbb{F}(\mathcal{H})_{n+1}+\partial_{n+2}\mathbb{F}(\mathcal{H})_{n+2}\big)\nonumber\\
&=&\partial_{n+1}  \mathbb{F}(\mathcal{H})_{n+1}. 
\label{eq-xyz-2}
\end{eqnarray}
By (\ref{eq-xyz-22}) and (\ref{eq-xyz-2}), we obtain (\ref{eq-xyz-3}).  
Furthermore, 
since  
\begin{eqnarray*}
\partial_{n+1} \mathbb{F}(\mathcal{H})_{n+1}  \subseteq \text{Ker}\partial_n\subseteq \partial_n^{-1} \mathbb{F}(\mathcal{H})_{n-1},
\end{eqnarray*}  
we have   
\begin{eqnarray}\label{eq-xyz-87}
\big(\text{Ker}\partial_n\cap\text{Inf}_n(\mathcal{H})\big)\cap\partial_{n+1} \mathbb{F}(\mathcal{H})_{n+1}  = \partial_{n+1} \mathbb{F}(\mathcal{H})_{n+1} \cap \mathbb{F}(\mathcal{H})_n. 
\end{eqnarray}
Therefore, by (\ref{eq-xyz-3}) and (\ref{eq-xyz-87}),  we have the orthogonal decomposition
\begin{eqnarray}
\text{Ker}\partial_n \cap \text{Sup}_n(\mathcal{H})&=&  \big(\text{Ker}\partial_n\cap\text{Inf}_n(\mathcal{H})\big) \nonumber\\
&&\oplus\perp\big(\partial_{n+1} \mathbb{F}(\mathcal{H})_{n+1} \cap \mathbb{F}(\mathcal{H})_n, \partial_{n+1} \mathbb{F}(\mathcal{H})_{n+1}\big). 
\label{eq-xyz-89}
\end{eqnarray}
By (\ref{eq-xyz-89}), we obtain (\ref{eq-xyz-6}). 
\end{proof}

The next theorem proves that the maps $i_2$ and  $i_3$ in the 
 diagram (\ref{diag-1}) are  isomorphisms.

\begin{theorem}\label{th-3.2-iso}
Let $\mathcal{H}$ be a hypergraph and $n\geq 0$. Then in the  diagram (\ref{diag-1}), the map $i_3$  is an isomorphism $\text{Ker}(L_n^{\text{Sup}_*(\mathcal{H})})\cap \text{Inf}_n(\mathcal{H})\overset{\cong}{\longrightarrow} \text{Ker}(L_n^{\text{Sup}_*(\mathcal{H})})$.  And the map $i_2$  is   an isomorphism $\text{Ker}(L_n^{\text{Sup}_*(\mathcal{H})})\cap \text{Inf}_n(\mathcal{H})\overset{\cong}{\longrightarrow} \text{Ker}(L_n^{\text{Inf}_*(\mathcal{H})})$.
\end{theorem}

\begin{proof}
Firstly, we study the map $i_3$. 
By  Theorem~\ref{pr.a.1} and (\ref{eq-xyz-3}),
\begin{eqnarray}
\text{Ker}(L_n^{\text{Sup}_*(\mathcal{H})})&=& \text{Ker}\partial_n\cap \text{Sup}_n(\mathcal{H})\cap \text{Ker}(\partial_{n+1}\mid_{\text{Sup}_*(\mathcal{H})})^*\nonumber\\
&=&\text{Ker}\partial_n\cap \text{Inf}_n(\mathcal{H})\cap  \text{Ker}(\partial_{n+1}\mid_{\text{Sup}_*(\mathcal{H})})^*\nonumber\\
&& + \partial_{n+1} \mathbb{F}(\mathcal{H})_{n+1}\cap  \text{Ker}(\partial_{n+1}\mid_{\text{Sup}_*(\mathcal{H})})^*,\label{eq-3.2.1.a}\\
\text{Ker}(L_n^{\text{Sup}_*(\mathcal{H})})\cap \text{Inf}_n(\mathcal{H})&=&  \text{Ker}\partial_n\cap \text{Inf}_n(\mathcal{H})\cap  \text{Ker}(\partial_{n+1}\mid_{\text{Sup}_*(\mathcal{H})})^*. \label{eq-3.2.1.b}
\end{eqnarray}
Since
\begin{eqnarray}
\text{Ker}(\partial_{n+1}\mid_{\text{Sup}_*(\mathcal{H})})^*&=&\perp \big( \text{Im} (\partial_{n+1}\mid_{\text{Sup}_*(\mathcal{H})}), \text{Sup}_n(\mathcal{H})\big)  \nonumber\\
&=&\perp\big( \partial_{n+1} \mathbb{F}(\mathcal{H})_{n+1}, \mathbb{F}(\mathcal{H})_n+ \partial_{n+1} \mathbb{F}(\mathcal{H})_{n+1}\big),
\label{eq-xyz-8}
\end{eqnarray}
we have
\begin{eqnarray}\label{eq-3.2.1.c}
\partial_{n+1} \mathbb{F}(\mathcal{H})_{n+1}\cap  \text{Ker}(\partial_{n+1}\mid_{\text{Sup}_*(\mathcal{H})})^*=0. 
\end{eqnarray}
Hence by (\ref{eq-3.2.1.a}), (\ref{eq-3.2.1.b}) and (\ref{eq-3.2.1.c}), $i_3$ is an isomorphism.

Secondly, we   study the map $i_2$.  By (\ref{eq-3.2.1.a}), (\ref{eq-3.2.1.b}) and (\ref{eq-xyz-8}), 
\begin{eqnarray}\label{eq-xyz-9}
\text{Ker}(L_n^{\text{Sup}_*(\mathcal{H})})&=& \perp\big(\partial_{n+1}{\text{Sup}_{n+1}(\mathcal{H})},\text{Sup}_n(\mathcal{H})\big) 
 \cap \text{Ker}\partial_n,\\
\text{Ker}(L_n^{\text{Sup}_*(\mathcal{H})})\cap\text{Inf}_n(\mathcal{H})&=& \perp\big(\partial_{n+1}{\text{Sup}_{n+1}(\mathcal{H})},\text{Sup}_n(\mathcal{H})\big) \cap \text{Ker}\partial_n\nonumber\\
&& \cap\text{Inf}_n(\mathcal{H}). \label{eq-xyz-99}
\end{eqnarray}
By a similar calculation with (\ref{eq-xyz-8}),  
\begin{eqnarray*} 
\text{Ker}(\partial_{n+1}\mid _{\text{Inf}_{n+1}(\mathcal{H})})^*=\perp\big(\partial_{n+1}{\text{Inf}_{n+1}(\mathcal{H})},\text{Inf}_n(\mathcal{H})\big).  
\end{eqnarray*}
Hence  by Theorem~\ref{pr.a.1}, 
\begin{eqnarray}\label{eq-xyz-10}
\text{Ker}(L_n^{\text{Inf}_*(\mathcal{H})})=\perp\big(\partial_{n+1}{\text{Inf}_{n+1}(\mathcal{H})},\text{Inf}_n(\mathcal{H})\big)
\cap  \text{Ker}\partial_n. 
\end{eqnarray}
We notice that the righthand side of (\ref{eq-xyz-9}) is canonically isomorphic to 
\begin{eqnarray*}
\big(\text{Ker}\partial_n \cap \text{Sup}_n(\mathcal{H})\big)\big/ \partial_{n+1} \text{Sup}_{n+1} (\mathcal{H}),
\end{eqnarray*}
 and the righthand side of (\ref{eq-xyz-10}) is canonically isomorphic to 
 \begin{eqnarray*}
\big(\text{Ker}\partial_n\cap\text{Inf}_n(\mathcal{H})\big)\big/ \partial_{n+1} \text{Inf}_{n+1} (\mathcal{H}). 
 \end{eqnarray*}
  Moreover, the righthand side of (\ref{eq-xyz-99}) is canonically isomorphic to 
  \begin{eqnarray*}
  \iota_*\big(\big(\text{Ker}\partial_n\cap\text{Inf}_n(\mathcal{H})\big)\big/ \partial_{n+1} \text{Inf}_{n+1} (\mathcal{H})\big).
  \end{eqnarray*}
    By the commutative diagram (\ref{diag-2}),  we see that $i_2$ in the diagram (\ref{diag-1}) is an isomorphism. 
\end{proof}

The next corollary follows from Theorem~\ref{th-3.2-iso}. 

\begin{corollary}
The diagram (\ref{diag-1})
 commutes. 
\end{corollary}

\begin{proof}
By Theorem~\ref{th-3.2-iso}, we see that the square
 \begin{eqnarray*}
\xymatrix{
 H_n(\mathcal{H};\mathbb{F})\ar[r]^{\cong} \ar [dd]^{\cong}&\text{Ker}(L_n^{\text{Sup}_*(\mathcal{H})}) \\
 \\
\text{Ker}(L_n^{\text{Inf}_*(\mathcal{H})})&\text{Ker}(L_n^{\text{Sup}_*(\mathcal{H})})\cap \text{Inf}_n(\mathcal{H})\ar[l]_{i_2\text{\ \ \ \ \ \ \ }}\ar[uu]^{i_3} 
 }
\end{eqnarray*}
 commutes.  With the help of Proposition~\ref{pr3.1.a.b.1} and Proposition~\ref{pr3.1.a.b.2},  
 the diagram (\ref{diag-1}) commutes. 
\end{proof}

\subsection{Functoriality and The Hodge Decompositions for   Hypergraphs}

In this subsection, we study the functoriality of the  decompositions in Theorem~\ref{th-decomp1}. We obtain the Hodge decomposition for hypergraphs in Theorem~\ref{th-3.19}.  

\smallskip 

Let $\mathcal{H}$ and $\mathcal{H}'$ be two hypergraphs and let $\rho: \mathcal{H}\longrightarrow \mathcal{H}'$ be a morphism.  Then $\rho$ is a map from the vertex-set of $\mathcal{H}$ to the vertex-set of $\mathcal{H}'$ such that for any hyperedge $\{v_0,\ldots,v_n\}$ of $\mathcal{H}$, $\{\rho(v_0),\ldots,\rho(v_n)\}$ is a hyperedge of $\mathcal{H}'$.  We have an induced simplicial map 
\begin{eqnarray*}
\Delta\rho: \Delta\mathcal{H}\longrightarrow \Delta\mathcal{H}' 
\end{eqnarray*}
sending a simplex $\{v_0,\ldots,v_k\}$ of $\Delta\mathcal{H}$ to a simplex $\{\rho(v_0),\ldots,\rho(v_k)\}$ of $\Delta\mathcal{H}$.  We have an induced homomorphism of homology groups
\begin{eqnarray*}
(\Delta\rho)_*: H_*(\Delta\mathcal{H};\mathbb{F})\longrightarrow H_*(\Delta\mathcal{H}';\mathbb{F}).  
\end{eqnarray*}
 Let $\overline{\Delta\rho}$ be the map sending  $\{v_0,\ldots,v_k\}$  to $\{\rho(v_0),\ldots,\rho(v_k)\}$ if $\rho(v_0)$, $\ldots$, $\rho(v_k)$ are distinct, and sending $\{v_0,\ldots,v_k\}$ to $0$ otherwise. By extending $\overline{\Delta\rho}$ linearly over $\mathbb{F}$,  we have a chain map
\begin{eqnarray*}
\mathbb{F}(\Delta\rho): \mathbb{F}(\Delta\mathcal{H})_*\longrightarrow \mathbb{F}(\Delta\mathcal{H}')_*. 
\end{eqnarray*}
And we have restricted chain maps
\begin{eqnarray}\label{eq-3.3.1.x}
\mathbb{F}(\Delta\rho)\mid_{\text{Inf}_*(\mathcal{H})}:&& \text{Inf}_*(\mathcal{H})\longrightarrow \text{Inf}_*(\mathcal{H}'),\\
\mathbb{F}(\Delta\rho)\mid_{\text{Sup}_*(\mathcal{H})}:&& \text{Sup}_*(\mathcal{H})\longrightarrow \text{Sup}_*(\mathcal{H}'). \label{eq-3.3.2.x}
\end{eqnarray}
By (\ref{eq-3.3.1.x}) and (\ref{eq-3.3.2.x}),  we have a commutative diagram of induced homomorphisms of the  homology groups
\begin{eqnarray*}
\xymatrix{
H_*(\text{Inf}_*(\mathcal{H}))\ar[rrr]^{(\mathbb{F}(\Delta\rho)\mid_{\text{Inf}_*(\mathcal{H})})_*}\ar[dd]_{\iota_*}^\cong&&&  H_*(\text{Inf}_*(\mathcal{H}'))\ar[dd]^\cong_{\iota'_*}\\
\\
H_*(\text{Sup}_*(\mathcal{H}))\ar[rrr]^{(\mathbb{F}(\Delta\rho)\mid_{\text{Sup}_*(\mathcal{H})})_*}&&&  H_*(\text{Sup}_*(\mathcal{H}')). 
}
\end{eqnarray*}
For simplicity, we denote the homomorphism between the embedded homology groups as 
\begin{eqnarray*}
\rho_*: H_*(\mathcal{H};\mathbb{F})\longrightarrow H_*(\mathcal{H}';\mathbb{F}). 
\end{eqnarray*}
Moreover, by restricting $\mathbb{F}(\Delta\rho)$ to $\mathbb{F}(\mathcal{H})_*$, we obtain a graded linear map
\begin{eqnarray*}
\mathbb{F}(\Delta\rho)\mid_{\mathbb{F}(\mathcal{H})_*}: \mathbb{F}(\mathcal{H})_* \longrightarrow \mathbb{F}(\mathcal{H}')_*. 
\end{eqnarray*}
By applying the maps $\rho_*$, $\Delta\rho$, $(\Delta\rho)_*$   and the restrictions of  $\mathbb{F}(\Delta\rho)$ on $\mathbb{F}(\mathcal{H})_*$, $\text{Inf}_*(\mathcal{H})$ and $\text{Sup}_*(\mathcal{H})$, we have the next theorem. 

\begin{theorem}\label{th-func}
The  decompositions in Theorem~\ref{th-decomp1}, Corollary~\ref{co-3.3.x} and Corollary~\ref{co-decomp2} are functorial. 
\qed
\end{theorem}

Suppose $\mathcal{H}\subseteq\mathcal{H}'$ and $\rho$ is the canonical inclusion of $\mathcal{H}$ into $\mathcal{H}'$.  Let $\partial'_*$ be the boundary maps of $\Delta\mathcal{H}'$.   Then  for $n\geq 0$, 
\begin{eqnarray}
 \text{Ker}(L^{\text{Inf}_*(\mathcal{H})}_n ) 
 =   \text{Ker}\partial'_n \cap \text{Inf}_n(\mathcal{H}) \cap \perp \big(\partial'_{n+1} \text{Inf}_{n+1}(\mathcal{H}), \text{Inf}_n(\mathcal{H})\big). \label{eq-3.4.87}
\end{eqnarray}
Moreover, the homomorphism $\rho_*$ sends   a chain    $\omega$ in (\ref{eq-3.4.87}) to  itself if $\omega\in  \text{Ker}(L^{\text{Inf}_*(\mathcal{H}')}_n )$, and sends $\omega$ to zero otherwise.

As a particular case, we let $\mathcal{H}'$ be $\Delta\mathcal{H}$ and let $s: \mathcal{H}\longrightarrow \Delta\mathcal{H}$ be the canonical inclusion.  Let $n\geq 0$.  Then $s$ induces a graded linear  map
\begin{eqnarray*}
s_\#:  \mathbb{F}(\mathcal{H})_*\longrightarrow \mathbb{F}(\Delta\mathcal{H})_* 
\end{eqnarray*}
and chain maps
\begin{eqnarray*}
s_\#^{\text{Inf}}:&&  \text{Inf}_*(\mathcal{H})\longrightarrow \mathbb{F}(\Delta\mathcal{H}),\\
s_\#^{\text{Sup}}:&& \text{Sup}_*(\mathcal{H})\longrightarrow \mathbb{F}(\Delta\mathcal{H}).
\end{eqnarray*}
The maps $s_\#$, $s_\#^{\text{Inf}}$ and $s_\#^{\text{Sup}}$ 
are  the canonical inclusions of vector spaces.  
With the help of   Theorem~\ref{th-a.1},   the induced homomorphism of the embedded homology satisfies the following commutative diagram
\begin{eqnarray}\label{diag-3.4.1}
\xymatrix{
H_n(\mathcal{H};\mathbb{F})\ar[dd]_{\cong }\ar[rr]^{s_*}&&H_n(\Delta\mathcal{H};\mathbb{F}) \ar[dd]^{\cong}\\
\\
 \text{Ker}(L^{\text{Inf}_*(\mathcal{H})}_n )  \ar@{-->}[rr]^{s^{\text{Inf}}_*}&& \text{Ker}(L_n^{\Delta\mathcal{H}})  \\
\\
&&\text{Ker}(L_n^{\Delta\mathcal{H}}) \cap \text{Inf}_n(\mathcal{H})\ar[uu]_f\ar[uull]_g. 
 }
\end{eqnarray}
Here $f$ and $g$ are the canonical inclusions.  Moreover, we have the following commutative diagram
\begin{eqnarray}\label{diag-3.4.2}
\xymatrix{
\text{Ker}(L_n^{\text{Sup}_*(\mathcal{H})})\ar@/_6.5pc/@{-->}[rrrrrdd]^{s_*^{\text{Sup}}}\ar[rr]\ar[dd]_{\cong} &&\text{Sup}_n(\mathcal{H})\ar[rrr]^{s_\#^{\text{Sup}}} &&&\mathbb{F}(\Delta\mathcal{H})_n\\
&&\mathbb{F}(\mathcal{H})_n\ar[u]\ar[rrru]^{s_\#}&&
\\
\text{Ker}(L_n^{\text{Inf}_*(\mathcal{H})})\ar@/_6pc/@{-->}[rrrrr]^{s_*^{\text{Inf}}}\ar[rr]\ar[ddd]_{\cong} &&\text{Inf}_n(\mathcal{H})\ar[u]\ar[rrruu]_{s_\#^{\text{Inf}}}&&&\text{Ker}(L_n^{\Delta\mathcal{H}})\ar[uu]\ar[ddd]^{\cong}\\
\\
\\
H_n(\mathcal{H};\mathbb{F})\ar[rrrrr]^{s_*}&&&&& H_n(\Delta\mathcal{H};\mathbb{F}).
}
\end{eqnarray}
All the unlabeled arrows in the diagram (\ref{diag-3.4.2}) are canonical inclusions.  And $s_*^{\text{Inf}}$ and $s_*^{\text{Sup}}$ are the corresponding maps of $s_*$ on the kernels of Laplacians.

Restricting $s_*$ to the two summands of the decomposition of $H_n(\mathcal{H};\mathbb{F})$ given in  Theorem~\ref{th-decomp1}~(a), the next theorem follows.

\begin{theorem} 
\label{th-3.18}
Let $\mathcal{H}$ be a hypergraph and $n\geq 0$. Then  
\begin{enumerate}[(a).]
\item
 the restriction of $s_*$ to 
 (\ref{eq-3.4.097}), denoted as 
 $
 s_*^1
$,
  is injective;
\item
the restriction of $s_*$ to 
(\ref{eq-3.4.098}), denoted as 
$s_*^{2} 
$,  is zero. 
\end{enumerate}
\end{theorem}

\begin{proof}
(a). It follows from  the commutative diagram (\ref{diag-3.4.1}) that  $s_*^1=s_*\circ g $ is injective.  

(b).  Let $\omega\in\text{Ker}(L_n^{\text{Sup}_*(\mathcal{H})})$.  By the top row of the commutative diagram (\ref{diag-3.4.2}),   
\begin{eqnarray*}
s_\#^{\text{Sup}} (\omega) \in s_\#\big(\text{Sup}_n(\mathcal{H})\big)=\text{Sup}_n(\mathcal{H}). 
\end{eqnarray*}
Hence
\begin{eqnarray*}
s_*^{\text{Sup}}(\omega)=s_\#^{\text{Sup}} (\omega)\in \text{Ker}(L^{\Delta\mathcal{H}}_n)\cap \text{Sup}_n(\mathcal{H}). 
\end{eqnarray*}
In particular, when $\omega$ be a chain in (\ref{eq-3.4.098}),  $s_*^{\text{Sup}}(\omega)=0$. 
    Therefore, $s_*^2$ is the zero map.   
  
  Alternatively, 
by Remark~\ref{re-81}, we see that as subspaces of $\mathbb{F}(\Delta\mathcal{H})_n$, (\ref{eq-3.4.098}) equals to the space 
\begin{eqnarray}\label{eq-3.4.198}
\perp\big( \text{Ker}(L_n^{\Delta\mathcal{H}})\cap\text{Inf}_n(\mathcal{H}), \text{Ker}(L_n^{\text{Inf}_*(\mathcal{H})})\big).  
\end{eqnarray}
 Let $\omega$ be a chain in (\ref{eq-3.4.198}).  
Then $\omega\in \text{Ker}(L_n^{\text{Inf}_*(\mathcal{H})})$ and $\omega\perp  \text{Ker}(L_n^{\Delta\mathcal{H}})\cap\text{Inf}_n(\mathcal{H})$.  By the maps $s^{\text{Inf}}_\#$ and $s^{\text{Inf}}_*$ in the diagram (\ref{diag-3.4.2}),  $s^{\text{Inf}}_*$ sends $\omega$ to zero. We also obtain that $s_*^2$ is the zero map. \end{proof}

Summarizing Theorem~\ref{th-decomp1} and Theorem~\ref{th-3.18}, we have the Hodge decompositions for hypergraphs. 

\begin{theorem}[Main Result I: Hodge Decompositions for Hypergraphs]\label{th-3.19}
Let $\mathcal{H}$ be a hypergraph and $n\geq 0$. Let $s$ be the canonical inclusion from $\mathcal{H}$ to $\Delta\mathcal{H}$ and $s_*$ be the induced homomorphism from $H_n(\mathcal{H};\mathbb{F})$ to $H_n(\Delta\mathcal{H};\mathbb{F})$. Then  represented by the kernel of the  supremum Laplacian $\text{Ker}(L_n^{\text{Sup}_*(\mathcal{H})})$, $H_n(\mathcal{H};\mathbb{F})$ is the orthogonal sum of  (\ref{eq-3.4.097}) and $\text{Ker}(s_*)$.  And represented by the kernel of the  Laplacian $\text{Ker}(L_n^{\Delta\mathcal{H}})$ of the simplicial complex $\Delta\mathcal{H}$,  $H_n(\Delta\mathcal{H};\mathbb{F})$ is the orthogonal sum of 
(\ref{eq-3.4.097}) and $\text{Coker}(s_*)$.  Moreover, $\text{Ker}(s_*)$ is given by (\ref{eq-3.4.098}) and  (IV)-(i), (IV)-(ii), (IV)-(iii); and  $\text{Coker}(s_*)$ is given by       (\ref{eq-3.4.876}) and (V)-(i), (V)-(ii), (V)-(iii).  \qed
\end{theorem}

The next corollary follows from Corollary~\ref{co-3.3.x}, Corollary~\ref{co-decomp2} and Theorem~\ref{th-3.19}. 

\begin{corollary}\label{co-3.99}
Let $\mathcal{H}$ be a hypergraph and $n\geq 0$. Then we have the orthogonal decompositions
\begin{eqnarray*}
\mathbb{F}(\Delta\mathcal{H})_n&=& \big(  \text{Ker}(L_n^{\Delta\mathcal{H}})\cap\text{Inf}_n(\mathcal{H}) \big)
\oplus \text{Coker}(s_*)\\
&&\oplus   \partial_{n+1}\big( {\mathbb{F}(\Delta\mathcal{H})_{n+1}}\big) 
  \oplus\partial^*_{n} \big(\mathbb{F}(\Delta\mathcal{H})_{n-1}\big) 
\end{eqnarray*}
and
\begin{eqnarray*}
\text{Sup}_n(\mathcal{H})&=& \big(  \text{Ker}(L_n^{\Delta\mathcal{H}})\cap\text{Inf}_n(\mathcal{H}) \big) 
 \oplus\text{Ker}(s_*)\\
&&\oplus  \partial_{n+1} \text{Sup}_{n+1}(\mathcal{H}) \oplus (\partial_{n}\mid _{\text{Sup}_*(\mathcal{H})})^* \text{Sup}_{n-1}(\mathcal{H}). 
\end{eqnarray*}
\qed
\end{corollary}

 \section{Hodge Decompositions for Weighted Hypergraphs} \label{sss4}

In this section, we generalize the Hodge decompositions for hypergraphs to the Hodge decompositions for  weighted hypergraphs.  We generalize Theorem~\ref{th-3.19} and obtain Theorem~\ref{th-0.0} (Theorem~\ref{th-4.19}).    We also discuss the relations between the weights on hypergraphs and the weighted embedded homology.

\subsection{Weighted Hypergraphs, Weighted Embedded Homology,  and Weighted Laplacians}\label{subs4.1}

In this subsection, we introduce the definitions of weighted hypergraphs, weighted embedded homology  and weighted Laplacians.

\smallskip

Let $n\geq 0$.  Let $\sigma$ be an $n$-simplex of $\Delta\mathcal{H}$.  For each $0\leq i\leq n$, let $d_i\sigma$ be the $(n-1)$-face by deleting the $i$-th vertex of $\sigma$.  We define  weights on $\mathcal{H}$ as the  weights on $\Delta\mathcal{H}$ (cf. \cite[Definition~2.1]{chengyuan}). The precise definition is given as follows. 

\begin{definition} \label{def1}
A {\bf weight} on $\mathcal{H}$ is a bilinear map $\phi: \mathbb{F}(\Delta\mathcal{H})\times \mathbb{F}(\Delta\mathcal{H})\longrightarrow \mathbb{F}$ such that 
\begin{eqnarray}\label{eq-xyz}
\phi(d_i\sigma,d_jd_i\sigma)\phi(\sigma,d_i\sigma)=\phi(d_j\sigma,d_jd_i\sigma)\phi(\sigma,d_j\sigma)
\end{eqnarray}
for any simplices $\sigma\in \Delta\mathcal{H}$ and any $j< i$.  
We call the pair $(\mathcal{H},\phi)$ a weighted hypergraph. 
\end{definition}

By  \cite[Definition~2.3]{chengyuan},  we have the  {\bf $\phi$-weighted boundary map} of $\Delta\mathcal{H}$
\begin{eqnarray*}
\partial^\phi_n: \mathbb{F}(\Delta\mathcal{H})_n\longrightarrow \mathbb{F}(\Delta\mathcal{H})_{n-1}
\end{eqnarray*}
given by
\begin{eqnarray*}
\partial^\phi_n(\sigma)=\sum_{i=0}^n(-1)^i\phi(\sigma,d_i\sigma)d_i\sigma. 
\end{eqnarray*}
By \cite[Proposition~2.5]{chengyuan}, 
$\partial^\phi_{n-1}\partial^\phi_n=0 
$.
 Thus we have a chain complex
\begin{eqnarray*}
\{\mathbb{F}(\Delta\mathcal{H})_n, \partial^\phi_n\}_{n\geq 0}. 
\end{eqnarray*}
Let $n\geq 0$. The $\phi$-weighted Laplacian  of $\Delta\mathcal{H}$ is
\begin{eqnarray*}
L^{\Delta\mathcal{H},\phi}_n=\partial^\phi_{n+1}(\partial^\phi_{n+1})^*+(\partial^\phi_n)^*\partial^\phi_n. 
\end{eqnarray*}
The $\phi$-weighted infimum chain complex and the $\phi$-weighted supremum chain complex are respectively
\begin{eqnarray*}
\text{Inf}^\phi_n(\mathcal{H})&=&\mathbb{F}(\mathcal{H})_n \cap (\partial^\phi_{n})^{-1} \mathbb{F}(\mathcal{H})_{n-1},\\
\text{Sup}^\phi_n(\mathcal{H})&=&\mathbb{F}(\mathcal{H})_n + \partial^{\phi}_{n+1} \mathbb{F}(\mathcal{H})_{n+1}.
\end{eqnarray*}
 We have the orthogonal decompositions 
\begin{eqnarray*}
\mathbb{F}(\mathcal{H})_n&=&\text{Inf}^\phi_n(\mathcal{H})\oplus A^\phi_n,\\
\text{Sup}^\phi_n(\mathcal{H})&=&\mathbb{F}(\mathcal{H})_n\oplus B^\phi_n,\\
\mathbb{F}(\Delta\mathcal{H})_n&=&\text{Sup}^\phi_n(\mathcal{H}) \oplus E^\phi_n 
\end{eqnarray*}
where $A^\phi_n$, $B^\phi_n$ and $D^\phi_n$ are  
\begin{eqnarray}
A^\phi_n&=&\perp \big(\mathbb{F}(\mathcal{H})_n\cap (\partial^\phi_n)^{-1}\mathbb{F}(\mathcal{H})_{n-1}, \mathbb{F}(\mathcal{H})_n\big), \label{eq-4q1}\\
B^\phi_n&=&\perp \big(\mathbb{F}(\mathcal{H})_n,\mathbb{F}(\mathcal{H})_n+ \partial^\phi_{n+1} \mathbb{F}(\mathcal{H})_{n+1}\big), \label{eq-4q2}\\
E^\phi_n&=&\perp \big(\mathbb{F}(\mathcal{H})_{n+1}+ \partial^\phi_{n+1} \mathbb{F}(\mathcal{H})_{n+1},\mathbb{F}(\Delta\mathcal{H})_n\big).   \label{eq-4q3}
\end{eqnarray}    
The $\phi$-weighted supremum Laplacian and the $\phi$-weighted infimum Laplacian of $\mathcal{H}$ are respectively  
\begin{eqnarray*}
L_n^{\text{Inf}^\phi_*(\mathcal{H}),\phi}&=&(\partial^\phi_{n+1}\mid_{\text{Inf}^\phi_*(\mathcal{H})})(\partial^\phi_{n+1}\mid_{\text{Inf}^\phi_*(\mathcal{H})})^* +(\partial^\phi_n\mid_{\text{Inf}^\phi_*(\mathcal{H})})^*(\partial^\phi_n\mid_{\text{Inf}^\phi_*(\mathcal{H})}),\\
L_n^{\text{Sup}^\phi_*(\mathcal{H}),\phi}&=&(\partial^\phi_{n+1}\mid_{\text{Sup}^\phi_*(\mathcal{H})})(\partial^\phi_{n+1}\mid_{\text{Sup}^\phi_*(\mathcal{H})})^* +(\partial^\phi_n\mid_{\text{Sup}^\phi_*(\mathcal{H})})^*(\partial^\phi_n\mid_{\text{Sup}^\phi_*(\mathcal{H})}). 
\end{eqnarray*}
The {$\phi$-weighted  embedded   homology}  of $\mathcal{H}$ is
\begin{eqnarray}
H_n(\mathcal{H},\phi;\mathbb{F})&=&H_n(\{\text{Inf}^\phi_*(\mathcal{H}),\partial_*^\phi\mid_{\text{Inf}^\phi_*(\mathcal{H})}\}) \nonumber\\
 &\cong& H_n(\{\text{Sup}^\phi_*(\mathcal{H}),\partial_*^\phi\mid_{\text{Sup}^\phi_*(\mathcal{H})}\}). 
 \label{eq3.1a.999}
\end{eqnarray} 
The isomorphism in (\ref{eq3.1a.999}) is obtained from \cite[Proposition~2.4]{h1}. 

\subsection{Some Examples}

In this subsection, we give some examples of weighted hypergraphs and their weighted embedded homology.

\smallskip

The next   example  shows that weighted simplicial complexes studied in \cite{chengyuan} is a special family of weighted hypergraphs. 

 \begin{example}
 Let $\mathcal{H}$ be a simplicial complex.  Then (\ref{eq3.1a.999}) gives  the $\phi$-weighted  homology of simplicial complexes. 
 The $\phi$-weighted (co)homology and the $\phi$-weighted Laplacian have been studied in \cite{chengyuan}. 
 \end{example}
 
 The next three examples give some  particular   kinds of weights for weighted  hypergraphs.  
 
 \begin{example}\label{ex-con}
 Suppose $\phi$  is given by
  \begin{eqnarray*}
  \phi(\sigma,\tau)=1
  \end{eqnarray*}  
  for any $\sigma,\tau\in\Delta\mathcal{H}$.  
  Then    $H_n(\mathcal{H},\phi;\mathbb{F})$ is the  embedded  homology of $\mathcal{H}$ studied in \cite{h1}.  
 \end{example}

 \begin{example}\label{ex-4.a.3}
  Suppose $\phi$  is given by
  \begin{eqnarray*}
  \phi(\sigma,\tau)=0
  \end{eqnarray*}  
  for any $\sigma,\tau\in\Delta\mathcal{H}$.  Let $n\geq 0$. Then $\partial_n^\phi=(\partial_n^\phi)^*=0$ and $\text{Inf}_n^\phi(\mathcal{H})= \text{Sup}_n^\phi(\mathcal{H})=\mathbb{F}(\mathcal{H})_n$.   Thus 
  \begin{eqnarray*}
  L^{\Delta\mathcal{H},\phi}_n=L_n^{\text{Inf}^\phi_*(\mathcal{H}),\phi}=L_n^{\text{Sup}^\phi_*(\mathcal{H}),\phi}=0 
  \end{eqnarray*}
and
   \begin{eqnarray*}
   H_n(\Delta\mathcal{H},\phi;\mathbb{F})&=&\mathbb{F}(\Delta\mathcal{H})_n,\\
   H_n(\mathcal{H},\phi;\mathbb{F})&=& \mathbb{F}(\mathcal{H})_n.  
  \end{eqnarray*}
 \end{example}

 \begin{example}\label{ex3.3.1} 
 Let 
 $w: \Delta\mathcal{H}\longrightarrow \mathbb{R}^{+}\subseteq\mathbb{F}
 $    be an evaluation function with positive real values on   the simplices of $\Delta\mathcal{H}$. 
 For any $\tau,\tau'\in\Delta\mathcal{H}$, let 
\begin{eqnarray}\label{eq3.3.0}
\phi_w(\sigma,\tau)=C\cdot \frac{w(\sigma)}{w(\tau)}. 
\end{eqnarray}
Here $C$ is a constant  positive real number which does not depend on the choices of $\sigma$ and $\tau$.  We extend $\phi_w$ bilinearly over $\mathbb{F}$. It is straightforward to verify that $\phi_w$ is a weight on $\mathcal{H}$.  In particular, when $\mathcal{H}$ is a simplicial complex, the weight $\phi_w$, the $\phi_w$-weighted (co)homology, and the $\phi_w$-weighted Laplacian have been studied in \cite{adv1}.    
\end{example}

The next example gives some concrete weighted hypergraphs and   calculations of the weighted embedded homology. 

\begin{example}\label{ex-4.1}
We consider the hypergraphs 
\begin{eqnarray*}
\mathcal{H}_0&=&\big\{\{v_0\},\{v_1\},\{v_2\},\{v_0,v_1,v_2\}\big\},\\
\mathcal{H}_1&=&\big\{\{v_0\},\{v_1\},\{v_2\},\{v_0,v_1\},\{v_0,v_1,v_2\}\big\},\\
\mathcal{H}_2&=&\big\{\{v_0\},\{v_1\},\{v_2\},\{v_0,v_1\},\{v_1,v_2\},\{v_0,v_1,v_2\}\big\},\\
\mathcal{H}_3&=&\big\{\{v_0\},\{v_1\},\{v_2\},\{v_0,v_1\},\{v_1,v_2\},\{v_0,v_2\},\{v_0,v_1,v_2\}\big\}.
\end{eqnarray*}
For each $\mathcal{H}_i$, $i=0,1,2,3$,  its associated simplicial complex  is $\mathcal{H}_3$.  These hypergraphs are drawn in Figure~\ref{fig2}. 
 
 \begin{figure}
  \begin{center}
\begin{tikzpicture}
\coordinate [label=left:$v_0$]  (A) at (2,0); 
\coordinate [label=right:$v_1$]  (B) at (5,0); 
   \coordinate [label=right:$\mathcal{H}_0$:]  (F) at (1,1); 
 \draw [dotted] (A) -- (B);
 \coordinate [label=right:$v_2$]  (C) at (4,1.5); 
 \draw [dotted] (B) -- (C);
  \draw [dotted] (A) -- (C);        
 \fill [gray!15!white] (A) -- (B) -- (C) -- cycle;
  \coordinate [label=left:$v_0$]  (A) at (9,0); 
  \coordinate [label=right:$v_1$]  (B) at (12,0); 
  \coordinate [label=right:$\mathcal{H}_1$:]  (F) at (8,1); 
 \draw [line width=2.5pt] (A) -- (B);
 \coordinate [label=right:$v_2$]  (C) at (11,1.5); 
 \draw[dotted] (C) -- (B);
\draw [dotted]   (A) -- (C);
\fill [gray!15!white]  (A) -- (B) -- (C) -- cycle;
\fill (2,0) circle (2.5pt) (5,0) circle (2.5pt)  (4,1.5) circle (2.5 pt);        
\fill (9,0) circle (2.5pt) (12,0) circle (2.5pt)  (11,1.5) circle (2.5 pt);        
\end{tikzpicture}
\end{center}
 \begin{center}
\begin{tikzpicture}
\coordinate [label=left:$v_0$]  (A) at (2,0); 
\coordinate [label=right:$v_1$]  (B) at (5,0); 
   \coordinate [label=right:$\mathcal{H}_2$:]  (F) at (1,1); 
 \draw [line width=2.5pt] (A) -- (B);
 \coordinate [label=right:$v_2$]  (C) at (4,1.5); 
 \draw [line width=2.5pt] (B) -- (C);
  \draw [dotted] (A) -- (C);        
 \fill [gray!15!white] (A) -- (B) -- (C) -- cycle;
  \coordinate [label=left: $v_0$]  (A) at (9,0); 
  \coordinate [label=right: $v_1$]  (B) at (12,0); 
  \coordinate [label=right: $\mathcal{H}_3$:]  (F) at (8,1); 
 \draw [line width=2.5pt] (A) -- (B);
 \coordinate [label=right:$v_2$]  (C) at (11,1.5); 
 \draw[line width=2.5pt] (C) -- (B);
\draw [line width=2.5pt]   (A) -- (C);
\fill [gray!15!white] (A) -- (B) -- (C) -- cycle;
\fill (2,0) circle (2.5pt) (5,0) circle (2.5pt)  (4,1.5) circle (2.5 pt);        
\fill (9,0) circle (2.5pt) (12,0) circle (2.5pt)  (11,1.5) circle (2.5 pt);        
\end{tikzpicture}
\end{center}
\caption{Example~\ref{ex-4.1}.}
\label{fig2}
\end{figure}
\begin{enumerate}[(a).]
\item
Let $\phi$ be a weight on $\mathcal{H}_3$ given by Definition~\ref{def1}. Then the $\phi$-weighted boundary map of $ \mathcal{H}_3$ is given by
\begin{eqnarray*}
    \partial_2^\phi(\{v_0,v_1,v_2\}) &=& \phi(\{v_0,v_1,v_2\},\{v_1,v_2\})  \{v_1,v_2\}\\
    && - \phi(\{v_0,v_1,v_2\}, \{v_0,v_2\}) \{v_0,v_2\} \\
   &&+\phi(\{v_0,v_1,v_2\}, \{v_0,v_1\}) \{v_0,v_1\},\\
   \partial_1^\phi(\{v_1,v_2\}) &=&  \phi(\{v_1,v_2\},  \{v_2\}) v_2  -  \phi(\{v_1,v_2\}, \{v_1\}) v_1,\\ 
     \partial_1^\phi(\{v_0,v_2\})  &=& \phi(\{v_0,v_2\},\{v_2\}) v_2  -  \phi(\{v_0,v_2\},\{v_0\}) v_0,\\ 
  \partial_1^\phi(\{v_0,v_1\}) &=&  \phi(\{v_0,v_1\},\{v_1\}) v_1  -  \phi(\{v_0,v_1\},\{v_0\}) v_0,  
 \end{eqnarray*}
 and $\partial_0^\phi(\{v_0\})=\partial_0^\phi(\{v_1\})=\partial_0^\phi(\{v_2\})=0$.   
 Hence for each $i=0,1,2,3$, 
 \begin{eqnarray*}
 H_2(\mathcal{H}_i,\phi;\mathbb{F})=
\left\{
\begin{array}{cc}
 0,  & \text{\ \ if } \phi(\{v_0,v_1,v_2\}, \{v_1,v_2\}), 
    \phi(\{v_0,v_1,v_2\},  \{v_0,v_2\}), \\
&\text{ and } \phi(\{v_0,v_1,v_2\}, \{v_0,v_1\}) 
 \text{  are not all zero};\\
 \mathbb{F},   & \text{\ \ if } \phi(\{v_0,v_1,v_2\}, \{v_1,v_2\}), 
    \phi(\{v_0,v_1,v_2\},  \{v_0,v_2\}), \\
&\text{ and } \phi(\{v_0,v_1,v_2\}, \{v_0,v_1\}) 
 \text{  are  all zero}. 
\end{array}
\right.
 \end{eqnarray*}
 Moreover, 
 \begin{eqnarray*}
 H_1(\mathcal{H}_i,\phi;\mathbb{F})=\big(\text{Ker}(\partial_1^\phi) \cap \mathbb{F}(\mathcal{H}_i)_1\big)\big/\big(\mathbb{F}\big(\partial_2^\phi(\{v_0,v_1,v_2\})\big) \cap  \mathbb{F}(\mathcal{H}_i)_1\big). 
 \end{eqnarray*}
Hence 
 \begin{eqnarray*}
 H_1(\mathcal{H}_0,\phi;\mathbb{F})=0,
  \end{eqnarray*}
   \begin{eqnarray*}
 H_1(\mathcal{H}_1,\phi;\mathbb{F})=\left\{
\begin{array}{cc}
  0, &\text{ if } \phi(\{v_0,v_1\},\{v_1\})   \text{ and }  \phi(\{v_0,v_1\},\{v_0\}) \text{ are not both zero};\\
  0, &\text{ if both } \phi(\{v_0,v_1\},\{v_1\})   \text{ and }  \phi(\{v_0,v_1\},\{v_0\}) \text{ are    zero},\\
  & \text{ both }\phi(\{v_0,v_1,v_2\},\{v_1,v_2\}) \\
  &   \text{ and } \phi(\{v_0,v_1,v_2\}, \{v_0,v_2\})  \text{ are zero},\\
     &\text{ and } \phi(\{v_0,v_1,v_2\}, \{v_0,v_1\}) \text{ is not zero}; \\
  \mathbb{F}, & \text{ if both } \phi(\{v_0,v_1\},\{v_1\})   \text{ and }  \phi(\{v_0,v_1\},\{v_0\}) \text{ are    zero},\\
  & \text{ and }  \phi(\{v_0,v_1,v_2\}, \{v_1,v_2\}), 
    \phi(\{v_0,v_1,v_2\},  \{v_0,v_2\}), \\
&\text{ and } \phi(\{v_0,v_1,v_2\}, \{v_0,v_1\}) 
 \text{  are  all zero};\\
\mathbb{F}, &  \text{ if both } \phi(\{v_0,v_1\},\{v_1\})   \text{ and }  \phi(\{v_0,v_1\},\{v_0\}) \text{ are    zero},\\
&\text{ and }\phi(\{v_0,v_1,v_2\},\{v_1,v_2\}), \\
&\phi(\{v_0,v_1,v_2\}, \{v_0,v_2\})  \text{ are not both zero}.  
  \end{array}
\right.
  \end{eqnarray*}
  The calculations of $H_1(\mathcal{H}_2,\phi;\mathbb{F})$ and  $H_1(\mathcal{H}_3,\phi;\mathbb{F})$ are similar to the calculation of  $H_1(\mathcal{H}_1,\phi;\mathbb{F})$. For simplicity, we omit the details.  
  Furthermore, 
 \begin{eqnarray*}
 H_0(\mathcal{H}_i,\phi;\mathbb{F})=\mathbb{F}(\{v_0\},\{v_1\},\{v_2\})/ \partial^\phi_1(\mathbb{F}(\mathcal{H}_i)_1). 
 \end{eqnarray*}
Hence 
 \begin{eqnarray*}
 H_0(\mathcal{H}_0,\phi;\mathbb{F})=\mathbb{F}^{\oplus 3},
  \end{eqnarray*}
  \begin{eqnarray*}
 H_0(\mathcal{H}_1,\phi;\mathbb{F})=\left\{
\begin{array}{cc}
  \mathbb{F}^{\oplus 3}, &\text{ if both } \phi(\{v_0,v_1\},\{v_1\})   \\
  &\text{ and }  \phi(\{v_0,v_1\},\{v_0\}) \text{ are zero};\\
  \mathbb{F}^{\oplus 2}, &\text{ if at least one of } \phi(\{v_0,v_1\},\{v_1\})\\
  &   \text{ and }  \phi(\{v_0,v_1\},\{v_0\}) \text{ is not zero}. 
  \end{array}
\right.
  \end{eqnarray*}
    The calculations of $H_0(\mathcal{H}_2,\phi;\mathbb{F})$ and  $H_0(\mathcal{H}_3,\phi;\mathbb{F})$ are similar to the calculation of  $H_0(\mathcal{H}_1,\phi;\mathbb{F})$. We omit the details.

 \item
Let $w: \mathcal{H}_3\longrightarrow (0, +\infty)$ be a function. Let $\phi_w$ be a weight on $\mathcal{H}_3$ induced from $w$ (cf. Example~\ref{ex3.3.1}).  The $\phi_w$-weighted boundary maps of  $\mathcal{H}_3$ are  
\begin{eqnarray*}
    \partial_2^{\phi_w}(\{v_0,v_1,v_2\}) &=& \frac{w(\{v_0,v_1,v_2\})}{w(\{v_1,v_2\})} \{v_1,v_2\} \\
    &&- \frac{w(\{v_0,v_1,v_2\})}{w(\{v_0,v_2\})} \{v_0,v_2\} \\
   &&+\frac{w(\{v_0,v_1,v_2\})}{w(\{v_0,v_1\})} \{v_0,v_1\},\\
   \partial_1^{\phi_w}(\{v_1,v_2\}) &=&  \frac{w(\{v_1,v_2\})}{w(\{v_2\})}v_2  -  \frac{w(\{v_1,v_2\})}{w(\{v_1\})}v_1,\\ 
     \partial_1^{\phi_w}(\{v_0,v_2\})  &=& \frac{w(\{v_0,v_2\})}{w(\{v_2\})}v_2  -  \frac{w(\{v_0,v_2\})}{w(\{v_0\})}v_0,\\ 
  \partial_1^{\phi_w}(\{v_0,v_1\}) &=&  \frac{w(\{v_0,v_1\})}{w(\{v_1\})}v_1  -  \frac{w(\{v_0,v_1\})}{w(\{v_0\})}v_0,  
 \end{eqnarray*}
 and $\partial_0^{\phi_w}(\{v_0\})=\partial_0^{\phi_w}(\{v_1\})=\partial_0^{\phi_w}(\{v_2\})=0$.   
Hence
\begin{eqnarray*}
&& H_2(\mathcal{H}_0,\phi_w;\mathbb{F})=H_1(\mathcal{H}_0,\phi_w;\mathbb{F})=0,\\
&& H_0(\mathcal{H}_0,\phi_w;\mathbb{F})=\mathbb{F}^{\oplus 3}; 
\end{eqnarray*}
\begin{eqnarray*}
&& H_2(\mathcal{H}_1,\phi_w;\mathbb{F})=H_1(\mathcal{H}_1,\phi;\mathbb{F})=0,\\
&& H_0(\mathcal{H}_1,\phi_w;\mathbb{F})=\mathbb{F}^{\oplus 2};  
\end{eqnarray*}
\begin{eqnarray*}
&& H_2(\mathcal{H}_2,\phi_w;\mathbb{F})=H_1(\mathcal{H}_2,\phi;\mathbb{F})=0,\\
&& H_0(\mathcal{H}_2,\phi_w;\mathbb{F})=\mathbb{F};  
\end{eqnarray*}
\begin{eqnarray*}
&& H_2(\mathcal{H}_3,\phi_w;\mathbb{F})=H_1(\mathcal{H}_3,\phi;\mathbb{F})=0,\\
&& H_0(\mathcal{H}_3,\phi_w;\mathbb{F})=\mathbb{F}. 
\end{eqnarray*}
The $\phi_w$-weighted embedded homology of $\mathcal{H}_0$, $\mathcal{H}_1$, $\mathcal{H}_2$ and $\mathcal{H}_3$ does not depend on  $w$. 
\end{enumerate}
\end{example}

\subsection{Relations Between   Weights and   Homology}

In this subsection,  we study the relations between the weights on hypergraphs and the weighted embedded homology. 

\smallskip

\begin{lemma}\label{le-4.3.1}
Let $\mathcal{H}$ be a hypergraph and $n\geq 0$.  Suppose $w: \Delta\mathcal{H}\longrightarrow (0,+\infty)$ is an evaluation function on $\Delta\mathcal{H}$ and   $\phi_w$ is the weight induced by $w$ in Example~\ref{ex3.3.1}.  Then
\begin{eqnarray*}
f_{n} :  \mathbb{F}(\mathcal{H})_n\cap \partial_{n+1}\mathbb{F}(\mathcal{H})_{n+1} \longrightarrow  \mathbb{F}(\mathcal{H})_n\cap \partial^{\phi_w}_{n+1}\mathbb{F}(\mathcal{H})_{n+1} 
\end{eqnarray*}
given by
\begin{eqnarray*}
f_n\big(\sum_{i=0}^{n+1} (-1)^i d_i\sigma\big)= \sum _{i=0}^{n+1} \frac{w(\sigma)}{w(d_i\sigma)}(-1)^i d_i\sigma
\end{eqnarray*}
is a linear isomorphism. 
\end{lemma}
\begin{proof}
The proof  is similar with \cite[Lemma~5.2]{chengyuan2}.   
\end{proof}

\begin{remark}\label{rem-4.3.1}
The  linear isomorphism in Lemma~\ref{le-4.3.1} can be generalized to  general  weights $\phi$ with nonzero values as follows.  
Let $\phi$ be a weight on $\mathcal{H}$ such that for any $\sigma\in \Delta\mathcal{H}$ with $\dim \sigma= n+1$, $\phi(\sigma,d_i\sigma)\neq 0$ for $0\leq i\leq n+1$. Then 
\begin{eqnarray*}
f_{n} :  \mathbb{F}(\mathcal{H})_n\cap \partial_{n+1}\mathbb{F}(\mathcal{H})_{n+1} \longrightarrow  \mathbb{F}(\mathcal{H})_n\cap \partial^{\phi}_{n+1}\mathbb{F}(\mathcal{H})_{n+1} 
\end{eqnarray*}
given by
\begin{eqnarray*}
f_n\big(\sum_{i=0}^{n+1} (-1)^i d_i\sigma\big)= \sum _{i=0}^{n+1} \phi(\sigma,d_i\sigma) (-1)^i d_i\sigma
\end{eqnarray*}
is a linear isomorphism.
\end{remark}

\begin{lemma}\label{le-4.3.2}
Let $\mathcal{H}$ be a hypergraph and $n\geq 0$.  Suppose $w: \Delta\mathcal{H}\longrightarrow (0,+\infty)$ is an evaluation function on $\Delta\mathcal{H}$ and   $\phi_w$ is the weight induced by $w$ in Example~\ref{ex3.3.1}.  Then
\begin{eqnarray*}
g_{n} :  \mathbb{F}(\mathcal{H})_n\cap \text{Ker}\partial_{n}  \longrightarrow  \mathbb{F}(\mathcal{H})_n\cap \text{Ker}\partial^{\phi_w}_{n } 
\end{eqnarray*}
given by
\begin{eqnarray*}
g_n\big(\sum_{k=1}^{m} a_k\sigma_k\big)= \sum _{k=1}^{m} \frac{a_k}{w( \sigma_k)} \sigma_k
\end{eqnarray*}
is a linear isomorphism. 
\end{lemma}
\begin{proof}
The proof  is similar with \cite[Lemma~5.1]{chengyuan2}.   
\end{proof}

The  linear isomorphism in Lemma~\ref{le-4.3.2} cannot be generalized to  general  weights $\phi$ with nonzero values or positive values. The following is such an example. 

\begin{example}
We consider the simplicial complex
\begin{eqnarray*}
\mathcal{K}  = \big\{\{v_0\},\{v_1\},\{v_2\},\{v_0,v_1\},\{v_1,v_2\},\{v_0,v_2\}\big\}.  
\end{eqnarray*}
Then 
\begin{eqnarray*}
\text{Ker}\partial_1= \mathbb{F}(\{v_1,v_2\}-\{v_0,v_2\}+\{v_0,v_1\}). 
\end{eqnarray*}
We consider a weight $\phi$ on $\mathcal{K}$ such that
\begin{eqnarray*}
&\phi(\{v_0,v_1\},\{v_0\}), &\phi(\{v_0,v_1\},\{v_1\}), \\
&\phi(\{v_1,v_2\},\{v_1\}), &\phi(\{v_1,v_2\},\{v_2\}), \\
&\phi(\{v_0,v_2\},\{v_0\}), &\phi(\{v_0,v_2\},\{v_2\}) 
\end{eqnarray*}
are positive and
\begin{eqnarray}\label{eq-4.3.1.c}
\frac{\phi(\{v_0,v_1\},\{v_0\})}{\phi(\{v_0,v_2\},\{v_0\})}\cdot \frac{\phi(\{v_0,v_2\},\{v_2\})}{\phi(\{v_1,v_2\},\{v_2\})}\cdot \frac{\phi(\{v_1,v_2\},\{v_1\})}{\phi(\{v_0,v_1\},\{v_1\})}\neq 1. 
\end{eqnarray}
  We prove that $\text{Ker}\partial_1^\phi=0$. 
Suppose to the contrary,  for some $a,b,c$ which are not all zero, 
\begin{eqnarray*}
\partial_1^\phi\big(a\{v_1,v_2\}-b\{v_0,v_2\}+c\{v_0,v_1\}\big) =0. 
\end{eqnarray*}
Then 
\begin{eqnarray*}
c \phi(\{v_0,v_1\},\{v_0\})&=&b  \phi(\{v_0,v_2\},\{v_0\}),\\
a \phi(\{v_1,v_2\},\{v_1\})  &=& c \phi(\{v_0,v_1\},\{v_1\}),\\
b \phi(\{v_0,v_2\},\{v_2\})  &=& a \phi(\{v_1,v_2\},\{v_2\}).  
\end{eqnarray*}
This contradicts with the assumption (\ref{eq-4.3.1.c}).  Hence $\text{Ker}\partial_1^\phi=0$, which is not isomorphic to  $\text{Ker}\partial_1$.  
\end{example} 

The next proposition follows from Lemma~\ref{le-4.3.1} and Lemma~\ref{le-4.3.2}. 

\begin{proposition}\label{th-iso1}
Let $\mathcal{H}$ be a hypergraph and $n\geq 0$.  Suppose $w: \Delta\mathcal{H}\longrightarrow (0,+\infty)$ is an evaluation function on $\Delta\mathcal{H}$ and   $\phi_w$ is induced by $w$ in Example~\ref{ex3.3.1}.  Then   
as vector spaces,  $H_n(\mathcal{H},\phi_w;\mathbb{F})\cong H_n(\mathcal{H};\mathbb{F})$.   
\end{proposition}

\begin{proof}
By an analogous calculation in \cite[Proposition~3.4]{h1}, Lemma~\ref{le-4.3.1} and Lemma~\ref{le-4.3.2}, 
\begin{eqnarray*}
H_n(\mathcal{H},\phi_w;\mathbb{F})&=&  \big(\mathbb{F}(\mathcal{H})_n\cap \text{Ker}\partial^{\phi_w}_{n} \big)\big/ \big(\mathbb{F}(\mathcal{H})_n\cap \partial^{\phi_w}_{n+1}\mathbb{F}(\mathcal{H})_{n+1} \big)\\
&\cong &  \big(\mathbb{F}(\mathcal{H})_n\cap \text{Ker}\partial_{n} \big)\big/ \big(\mathbb{F}(\mathcal{H})_n\cap \partial_{n+1}\mathbb{F}(\mathcal{H})_{n+1} \big)\\
&=& H_n(\mathcal{H};\mathbb{F}). 
\end{eqnarray*}
\end{proof}

The next proposition follows from Remark~\ref{rem-4.3.1} and the proof of Proposition~\ref{th-iso1}. 

\begin{proposition}
Let $\mathcal{H}$ be a hypergraph and $n\geq 0$. Let $\phi$ be a weight on $\mathcal{H}$ such that for any 
  $\sigma\in \Delta\mathcal{H}$ with $\dim \sigma= n+1$, $\phi(\sigma,d_i\sigma)\neq 0$ for $0\leq i\leq n+1$.  If 
  as vector spaces, $
  \mathbb{F}(\mathcal{H})_n\cap \text{Ker}\partial^{\phi}_{n} \cong\mathbb{F}(\mathcal{H})_n\cap \text{Ker}\partial_{n} 
$, 
then as vector spaces, $H_n(\mathcal{H},\phi;\mathbb{F})\cong H_n(\mathcal{H};\mathbb{F})$.   
\qed
\end{proposition}

We consider the weighted infimum chain complex   in the next proposition. 

\begin{proposition}\label{le-4.3.8}
Let $\mathcal{H}$ be a hypergraph and $n\geq 0$.  Suppose $w: \Delta\mathcal{H}\longrightarrow (0,+\infty)$ is an evaluation function on $\Delta\mathcal{H}$ and   $\phi_w$ is the weight induced by $w$ in Example~\ref{ex3.3.1}.  Then $\text{Inf}^{\phi_w}_n(\mathcal{H})=\text{Inf}_n(\mathcal{H})$.
\end{proposition}

\begin{proof}
By the proof of \cite[Lemma~5.1,    equations (5.1)-(5.3)]{chengyuan2},   
\begin{eqnarray*}
&\partial_n\big(\sum_{k=1}^{m} a_k\sigma_k\big)\in \mathbb{F}(\mathcal{H})_{n-1}\\
 \Longleftrightarrow& d_j\sigma_k\in \mathcal{H} \text{ and }  \dim (d_j\sigma_k)=n-1\\
 & \text{ for any } 1\leq k\leq m,  0\leq j\leq n\\
 \Longleftrightarrow & \partial_n^{\phi_w}  \big(\sum_{k=1}^{m} a_k\sigma_k\big)\in \mathbb{F}(\mathcal{H})_{n-1}.
\end{eqnarray*}
Hence 
\begin{eqnarray}\label{eq-4.3.6}
\partial_n^{-1}\mathbb{F}(\mathcal{H})_{n-1}=  (\partial^{\phi_w}_n)^{-1}\mathbb{F}(\mathcal{H})_{n-1}.
\end{eqnarray} 
On both sides of (\ref{eq-4.3.6}), taking the intersections with $\mathbb{F}(\mathcal{H})_n$, we obtain the assertion.  
\end{proof}

 \begin{remark}
 By the proof of \cite[Lemma~5.2]{chengyuan2},  under the conditions of Proposition~\ref{le-4.3.8},    we have  a linear isomorphism 
 \begin{eqnarray}\label{eq-4.3.91}
 \partial^{\phi_w}_{n+1}(\mathbb{F}(\mathcal{H})_{n+1})\cong   \partial_{n+1}(\mathbb{F}(\mathcal{H})_{n+1}). 
 \end{eqnarray}
Moreover, 
  \begin{eqnarray}
 \text{Sup}_n^{\phi_w}(\mathcal{H})&=&\mathbb{F}(\mathcal{H})_n + \partial^{\phi_w}_{n+1}(\mathbb{F}(\mathcal{H})_{n+1})\nonumber\\
 &\cong& \mathbb{F}(\mathcal{H})_n +  \partial_{n+1}(\mathbb{F}(\mathcal{H})_{n+1})\label{eq-4.3.92}\\
 &=& \text{Sup}_n(\mathcal{H}). \nonumber
 \end{eqnarray}
 The linear isomorphisms in (\ref{eq-4.3.91}) and (\ref{eq-4.3.92}) may not be identity maps. That is, as subspaces of $\mathbb{F}(\Delta\mathcal{H})_n$,   $\partial^{\phi_w}_{n+1}(\mathbb{F}(\mathcal{H})_{n+1})$ and   $ \partial_{n+1}(\mathbb{F}(\mathcal{H})_{n+1})$ may not be equal; and $\text{Sup}_n^{\phi_w}(\mathcal{H})$ and $\text{Sup}_n(\mathcal{H})$ may not be equal. 

 \end{remark}
 
  \subsection{The Hodge Decompositions for Weighted Hypergraphs}

 In this subsection,  we study  the Hodge decompositions for weighted hypergraphs and prove the  main result Theorem~\ref{th-4.19}.  

\smallskip
 
Theorem~\ref{pr.a.1} can be generalized to weighted hypergraphs in the next theorem. 
 
\begin{theorem}[Hodge Isomorphism for Weighted Hypergraphs I]\label{pr.a.1.w}
Let $\mathcal{H}$ be a hypergraph.  Let $\phi$ be a weight on $\mathcal{H}$. For each $n\geq 0$, 
\begin{eqnarray*}
H_n(\mathcal{H},\phi;\mathbb{F}) \cong   \text{Ker}(L_n^{\text{Inf}^\phi_*(\mathcal{H}), \phi})  
 \cong  \text{Ker}(L_n^{\text{Sup}^\phi_*(\mathcal{H}),\phi}). 
\end{eqnarray*}
In other words, 
\begin{eqnarray*}
H_n(\mathcal{H},\phi;\mathbb{F})&\cong&  \text{Ker}(\partial^\phi_n\mid_{\text{Inf}^\phi_*(\mathcal{H})})\cap \text{Ker}(\partial^\phi_{n+1}\mid_{\text{Inf}^\phi_*(\mathcal{H})})^* \\
 &\cong& \text{Ker}(\partial^\phi_n\mid_{\text{Sup}^\phi_*(\mathcal{H})})\cap \text{Ker}(\partial^\phi_{n+1}\mid_{\text{Sup}^\phi_*(\mathcal{H})})^*.
\end{eqnarray*}
\qed
\end{theorem}

Theorem~\ref{th-a.1} can be generalized to weighted hypergraphs in the next theorem.

\begin{theorem}[Hodge Isomorphism for Weighted Hypergraphs II]\label{th-a.1.w}
Let $\mathcal{H}$ be a hypergraph, $\phi$ a weight on $\mathcal{H}$, and   $n\geq 0$.   Then both $\text{Ker}L^{\Delta\mathcal{H},\phi}_n\cap \text{Inf}^\phi_n(\mathcal{H})$ and $\text{Ker}L^{\Delta\mathcal{H},\phi}_n\cap \text{Sup}^\phi_n(\mathcal{H})$ are   subspaces of $H_n(\mathcal{H},\phi;\mathbb{F})$.  Moreover, {if}  $\partial_n(A^\phi_n\oplus B^\phi_n\oplus E^\phi_n)\subseteq A^\phi_{n-1}\oplus B^\phi_{n-1}\oplus E^\phi_{n-1}$, then 
\begin{eqnarray*}
\text{Ker}L^{\Delta\mathcal{H},\phi}_n\cap \text{Inf}^\phi_n(\mathcal{H})\cong H_n(\mathcal{H},\phi;\mathbb{F}). 
\end{eqnarray*}
And {if} 
$\partial_n( E^\phi_n)\subseteq E^\phi_{n-1}$,  then 
\begin{eqnarray*}
\text{Ker}L^{\Delta\mathcal{H},\phi}_n\cap \text{Sup}^\phi_n(\mathcal{H})\cong H_n(\mathcal{H},\phi;\mathbb{F}). 
\end{eqnarray*}
\qed
\end{theorem}

With the help of Theorem~\ref{pr.a.1.w} and Theorem~\ref{th-a.1.w},  Theorem~\ref{th-decomp1} can be generalized to the next theorem. 

\begin{theorem}
\label{th-decomp123}
Let $\mathcal{H}$ be a hypergraph, $\phi$ a weight on $\mathcal{H}$, and   $n\geq 0$. Then we have 
\begin{enumerate}[(a).]
\item
the orthogonal decomposition of the $\phi$-weighted embedded homology into two summands
\begin{eqnarray*}
H_n(\mathcal{H},\phi;\mathbb{F})&\cong & \big( \text{Ker}(L_n^{\Delta\mathcal{H},\phi})\cap\text{Inf}^\phi_n(\mathcal{H}) \big)\\
&&\oplus\perp\big( \text{Ker}(L_n^{\Delta\mathcal{H},\phi})\cap\text{Sup}^\phi_n(\mathcal{H}), \text{Ker}(L_n^{\text{Sup}^\phi_*(\mathcal{H}),\phi})\big); 
\end{eqnarray*}
\item
 the orthogonal decomposition of the homology of $\Delta\mathcal{H}$ into two summands
\begin{eqnarray*}
H_n(\Delta\mathcal{H},\phi;\mathbb{F})&\cong & \big( H_n(\Delta\mathcal{H},\phi;\mathbb{F})\cap\text{Inf}^\phi_n(\mathcal{H}) \big)\\
&&\oplus\perp\big( \text{Ker}(L_n^{\Delta\mathcal{H},\phi})\cap\text{Sup}^\phi_n(\mathcal{H}), \text{Ker}(L_n^{\Delta\mathcal{H},\phi})\big). 
\end{eqnarray*}
\end{enumerate}
\qed
\end{theorem}

Generalizing Corollary~\ref{co-3.3.x} to weighted hypergraphs, the next corollary follows from Theorem~\ref{th-decomp123}~(b). 

\begin{corollary}\label{co-3.3.x.a}
Let $\mathcal{H}$ be a hypergraph, $\phi$ a weight on $\mathcal{H}$, and   $n\geq 0$.  Then we have the orthogonal decomposition of the vector space spanned by the $n$-simplices of $\Delta\mathcal{H}$ into four summands
\begin{eqnarray*}
\mathbb{F}(\Delta\mathcal{H})_n&=& \big( \text{Ker}(L_n^{\Delta\mathcal{H},\phi})\cap\text{Inf}^\phi_n(\mathcal{H}) \big)\\
&&\oplus\perp\big( \text{Ker}(L_n^{\Delta\mathcal{H},\phi})\cap\text{Sup}^\phi_n(\mathcal{H}), \text{Ker}(L_n^{\Delta\mathcal{H},\phi})\big)\\
&&\oplus   \partial^\phi_{n+1}( {\mathbb{F}(\Delta\mathcal{H})_{n+1}}) 
  \oplus(\partial^\phi_{n})^* (\mathbb{F}(\Delta\mathcal{H})_{n-1}). 
\end{eqnarray*}
\qed
\end{corollary}

Generalizing Corollary~\ref{co-decomp2} to weighted hypergraphs, the next corollary follows from Theorem~\ref{th-decomp123}~(a). 
 
\begin{corollary}\label{co-decomp2.a}
Let $\mathcal{H}$ be a hypergraph, $\phi$ a weight on $\mathcal{H}$, and   $n\geq 0$.  Then we have the orthogonal decomposition of the $n$-dimensional space of the $\phi$-weighted supremum chain complex into four summands
\begin{eqnarray*}
\text{Sup}^\phi_n(\mathcal{H})&=& \big( \text{Ker}(L_n^{\Delta\mathcal{H},\phi})\cap\text{Inf}^\phi_n(\mathcal{H}) \big)\\
&&\oplus\perp\big( \text{Ker}(L_n^{\Delta\mathcal{H},\phi})\cap\text{Sup}^\phi_n(\mathcal{H}), \text{Ker}(L_n^{\text{Sup}^\phi_*(\mathcal{H}),\phi})\big)\\
&&\oplus  \partial^\phi_{n+1} \text{Sup}^\phi_{n+1}(\mathcal{H}) \oplus (\partial^\phi_{n}\mid _{\text{Sup}^\phi_*(\mathcal{H})})^* \text{Sup}^\phi_{n-1}(\mathcal{H}). 
\end{eqnarray*}
\qed
\end{corollary}

 \begin{example}
 Let $\mathcal{H}$ be a simplicial complex.  Then  Theorem~\ref{pr.a.1.w} and Corollary~\ref{co-3.3.x.a} (or equivalently, Corollary~\ref{co-decomp2.a}) are reduced to the Hodge isomorphisms and Hodge decompositions of weighted simplicial complexes (cf. \cite{chengyuan}) respectively.  And Theorem~\ref{th-a.1.w} and Theorem~\ref{th-decomp123} are reduced to the trivial statements. 
 \end{example}
 
 \begin{example}\label{ex-con.ab}
 Suppose $\phi$  is given by
$  \phi(\sigma,\tau)=1
$ 
  for any $\sigma,\tau\in\Delta\mathcal{H}$.  
  Then   Theorem~\ref{pr.a.1.w}, Theorem~\ref{th-a.1.w}, Theorem~\ref{th-decomp123}, Corollary~\ref{co-3.3.x.a} and Corollary~\ref{co-decomp2.a} are reduced to Theorem~\ref{pr.a.1}, Theorem~\ref{th-a.1}, Theorem~\ref{th-decomp1}, Corollary~\ref{co-3.3.x} and Corollary~\ref{co-decomp2} respectively. 
 \end{example}

Let $(\mathcal{H},\phi)$ and $(\mathcal{H}',\phi')$ be two weighted hypergraphs.  A morphism of weighted hypergraphs is a morphism of hypergraphs $\rho: \mathcal{H}\longrightarrow \mathcal{H}'$ such that for any $n\geq 0$, the following diagram commutes
\begin{eqnarray*}
\xymatrix{
\mathbb{F}(\Delta\mathcal{H})_{n+1}\ar[rr]^{\mathbb{F}(\Delta\rho)}\ar[dd]^{\partial_n^\phi}&& \mathbb{F}(\Delta\mathcal{H}')_{n+1}\ar[dd]^{{\partial'}_{n}^{\phi'}}\\
\\
\mathbb{F}(\Delta\mathcal{H})_n\ar[rr]^{\mathbb{F}(\Delta\rho)}&& \mathbb{F}(\Delta\mathcal{H}')_n. 
}
\end{eqnarray*}
Here $\partial_n^\phi$ is the $\phi$-weighted boundary map of $\Delta\mathcal{H}$ and ${\partial'}_{n}^{\phi'}$ is the $\phi'$-weighted boundary map of $\Delta\mathcal{H}'$.  
\begin{theorem}\label{th-func123}
The  decompositions in Theorem~\ref{th-decomp123}, Corollary~\ref{co-3.3.x.a} and Corollary~\ref{co-decomp2.a} are functorial. 
\qed
\end{theorem}

\begin{remark}
We consider the canonical inclusion $s: \mathcal{H}\longrightarrow \Delta\mathcal{H}$. Then $s$ is a morphism of weighted hypergraphs  from $(\mathcal{H},\phi)$ to $(\Delta\mathcal{H},\phi)$. 
\end{remark}

\begin{remark}
 In particular,  when $\mathcal{H}$ is a simplicial complex,  $H^*(\mathcal{H},\phi;\mathbb{F})$,  the cohomology version of $H_*(\mathcal{H},\phi;\mathbb{F})$,  is studied in \cite{chengyuan}.
\end{remark}

With the help of Theorem~\ref{th-decomp123} and Theorem~\ref{th-func123},  Theorem~\ref{th-3.19} can be generalized to weighted hypergraphs. 

\begin{theorem}[Main Result II: Hodge Decompositions for Weighted Hypergraphs]\label{th-4.19}
Let $\mathcal{H}$ be a hypergraph, $\phi$ a weight on $\mathcal{H}$, and $n\geq 0$. Let $s$ be the canonical inclusion from $\mathcal{H}$ to $\Delta\mathcal{H}$ and $s_*$ be the induced homomorphism from $H_n(\mathcal{H},\phi;\mathbb{F})$ to $H_n(\Delta\mathcal{H},\phi;\mathbb{F})$. Then  represented by the kernel of the weighted  supremum Laplacian $\text{Ker}(L_n^{\text{Sup}^\phi_*(\mathcal{H}),\phi})$, $H_n(\mathcal{H},\phi;\mathbb{F})$ is the orthogonal sum of  $\text{Ker}(L_n^{\Delta\mathcal{H},\phi})\cap\text{Inf}^\phi_n(\mathcal{H})$ and $\text{Ker}(s_*)$.  And represented by the kernel of the weighted  Laplacian $\text{Ker}(L_n^{\Delta\mathcal{H},\phi})$,  $H_n(\Delta\mathcal{H},\phi;\mathbb{F})$ is the orthogonal sum of 
$\text{Ker}(L_n^{\Delta\mathcal{H},\phi})\cap\text{Inf}^\phi_n(\mathcal{H})$ and $\text{Coker}(s_*)$.  
 \qed
\end{theorem}

The next corollary follows from Corollary~\ref{co-3.3.x.a}, Corollary~\ref{co-decomp2.a} and Theorem~\ref{th-4.19}. 

\begin{corollary}\label{co-4.99}
Let $\mathcal{H}$ be a hypergraph, $\phi$ a weight on $\mathcal{H}$, and $n\geq 0$. Then we have the orthogonal decompositions
\begin{eqnarray*}
\mathbb{F}(\Delta\mathcal{H})_n&=& \big(  \text{Ker}(L_n^{\Delta\mathcal{H},\phi})\cap\text{Inf}^\phi_n(\mathcal{H}) \big)
\oplus \text{Coker}(s_*)\\
&&\oplus   \partial^\phi_{n+1}\big( {\mathbb{F}(\Delta\mathcal{H})_{n+1}}\big) 
  \oplus(\partial^\phi_{n})^* \big(\mathbb{F}(\Delta\mathcal{H})_{n-1}\big) 
\end{eqnarray*}
and
\begin{eqnarray*}
\text{Sup}^\phi_n(\mathcal{H})&=& \big(  \text{Ker}(L_n^{\Delta\mathcal{H},\phi})\cap\text{Inf}^\phi_n(\mathcal{H}) \big) 
 \oplus\text{Ker}(s_*)\\
&&\oplus  \partial^\phi_{n+1} \text{Sup}^\phi_{n+1}(\mathcal{H}) \oplus (\partial^\phi_{n}\mid _{\text{Sup}^\phi_*(\mathcal{H})})^* \text{Sup}^\phi_{n-1}(\mathcal{H}). 
\end{eqnarray*}
\qed
\end{corollary}

\begin{remark}
In Theorem~\ref{th-4.19} and Corollary~\ref{co-4.99}, $s_*$ depends on $\phi$. Hence $\text{Ker}(s_*)$ and $\text{Coker}(s_*)$  depend on $\phi$ as well. 
\end{remark}

The next theorem follows by applying   \cite[Theorem~5.3]{chengyuan2},  Proposition~\ref{th-iso1} and Proposition~\ref{le-4.3.8} to Theorem~\ref{th-4.19} and Corollary~\ref{co-4.99}.  

\begin{theorem}\label{th-iso11111}
Let $\mathcal{H}$ be a hypergraph and $n\geq 0$.  Suppose $w: \Delta\mathcal{H}\longrightarrow (0,+\infty)$ is an evaluation function on $\Delta\mathcal{H}$ and   $\phi_w$ is induced by $w$ in Example~\ref{ex3.3.1}.  Then  the orthogonal decompositions in Theorem~\ref{th-4.19} and Corollary~\ref{co-4.99} are the same as the orthogonal decompositions in Theorem~\ref{th-3.19}  and Corollary~\ref{co-3.99} respectively.  
\qed
\end{theorem}

 \begin{remark}
By Theorem~\ref{th-iso11111}, the kernels  of $L_n^{\Delta\mathcal{H},\phi_w}$, $L_n^{\text{Inf}^{\phi_w}_*(\mathcal{H}),\phi_w}$ and $L_n^{\text{Sup}^{\phi_w}_*(\mathcal{H}),\phi_w}$ do not 
depend on $\phi_w$.   Nevertheless, the eigenvalues of $L_n^{\Delta\mathcal{H},\phi_w}$, $L_n^{\text{Inf}^{\phi_w}_*(\mathcal{H}),\phi_w}$ and $L_n^{\text{Sup}^{\phi_w}_*(\mathcal{H}),\phi_w}$ 
may depend on $\phi_w$.   In particular, when $\mathcal{H}$ is a simplicial complex, these eigenvalues are studied in \cite{adv1}. 
 \end{remark}

\section{Eigenvalues of The Weighted Laplacians of Weighted Hypergraphs 
}\label{sec-a}

 In this section,   we study the nonzero eigenvalues of the weighted Laplacians for  weighted hypergraphs.  
   
 \smallskip
 
 Let $(\mathcal{H},\phi)$ be a weighted hypergraph.   Let
 \begin{eqnarray*}
 (L_n^{\text{Inf}^\phi_*(\mathcal{H}),\phi})^\text{up}&=&(\partial^\phi_n\mid_{\text{Inf}^\phi_*(\mathcal{H})})^*(\partial^\phi_n\mid_{\text{Inf}^\phi_*(\mathcal{H})}),\\
 (L_n^{\text{Inf}^\phi_*(\mathcal{H}),\phi})^\text{down}&=&  (\partial^\phi_{n+1}\mid_{\text{Inf}^\phi_*(\mathcal{H})})(\partial^\phi_{n+1}\mid_{\text{Inf}^\phi_*(\mathcal{H})})^*,\\
  (L_n^{\text{Sup}^\phi_*(\mathcal{H}),\phi})^\text{up}&=&(\partial^\phi_n\mid_{\text{Sup}^\phi_*(\mathcal{H})})^*(\partial^\phi_n\mid_{\text{Sup}^\phi_*(\mathcal{H})}),\\
 (L_n^{\text{Sup}^\phi_*(\mathcal{H}),\phi})^\text{down}&=&  (\partial^\phi_{n+1}\mid_{\text{Sup}^\phi_*(\mathcal{H})})(\partial^\phi_{n+1}\mid_{\text{Sup}^\phi_*(\mathcal{H})})^*,\\
  (L_n^{\Delta\mathcal{H},\phi})^\text{up}&=& (\partial^\phi_n)^*(\partial^\phi_n),\\
 (L_n^{\Delta\mathcal{H},\phi})^\text{down}&=& (\partial^\phi_{n+1})(\partial^\phi_{n+1})^*. 
 \end{eqnarray*}

For any linear operator $A$ acting on a (finite dimensional) vector space, we denote the weakly increasing rearrangement of its eigenvalues, together with the corresponding multiplicities,  by ${\bf{s}}(A)$. We write ${\bf{s}}(A)\overset{\circ}{=}{\bf{s}}(B)$ if the multisets ${\bf{s}}(A)$ and ${\bf{s}}(B)$ differ only in their multiplicities of zero (cf. \cite[p. 308]{adv1}).  We write ${\bf{s}}(A) \subseteq {\bf{s}}(B)$ if the multiset  ${\bf{s}}(A)$ is contained in ${\bf{s}}(B)$, i.e., each eigenvalue $\lambda$ of $A$ is an eigenvalue of $B$, and the multiplicity of $\lambda$  as an eigenvalue of  $A$ is smaller than or equal to the multiplicity of $\lambda$ as an eigenvalue of $B$. 
Moreover, we write ${\bf{s}}(A)\overset{\circ}{\subseteq}{\bf{s}}(B)$ if   ${\bf{s}}(A)$ is contained in ${\bf{s}}(B)$ except for   the    multiplicities of the eigenvalue zero.  We denote the union of multisets by $\overset{\circ}{\cup}$.   The next proposition follows by a similar argument of \cite[p. 308, (2.5)]{adv1}. 

\begin{proposition}\label{pr-5.1}
 Let $(\mathcal{H},\phi)$ be a weighted hypergraph and $n\geq 0$. Then 
\begin{enumerate}[(a).]
\item
${\bf{s}}\big( L_n^{\text{Inf}^\phi_*(\mathcal{H}),\phi}\big) \overset{\circ}{=}{\bf{s}}\big( (L_n^{\text{Inf}^\phi_*(\mathcal{H}),\phi})^{\text{up} }\big)\overset{\circ}{\cup} {\bf{s}}\big( (L_n^{\text{Inf}^\phi_*(\mathcal{H}),\phi})^{\text{down} }\big)$, 
\item
${\bf{s}}\big( L_n^{\text{Sup}^\phi_*(\mathcal{H}),\phi}\big) \overset{\circ}{=}{\bf{s}}\big( (L_n^{\text{Sup}^\phi_*(\mathcal{H}),\phi})^{\text{up} }\big)\overset{\circ}{\cup} {\bf{s}}\big( (L_n^{\text{Sup}^\phi_*(\mathcal{H}),\phi})^{\text{down} }\big)$, 
\item
${\bf{s}}\big( L_n^{\Delta\mathcal{H},\phi}\big) \overset{\circ}{=}{\bf{s}}\big( (L_n^{\Delta\mathcal{H},\phi})^{\text{up} }\big)\overset{\circ}{\cup} {\bf{s}}\big( (L_n^{\Delta\mathcal{H},\phi})^{\text{down} }\big)$. 
\end{enumerate}
\qed
\end{proposition}

Let $T$ be a linear operator on an Euclidean space $W$. Let $W'$ be a subspace of $W$.  We use $T||_{W'}$ to denote the restriction of $T$ on $W'$. Then $T||_{W'}$ is a map from $W'$ to $W$.  Here we do not require $W'$ to be a $T$-invariant subspace, hence the image of $T||_{W'}$ may not be contained in $W'$.  We say that $\lambda$ is an quasi-eigenvalue of $T||_{W'}$ if there exists a nonzero vector $v\in W'$ such that $Tv =\lambda v$.  We use the term quasi-eigenvalue for the reason that $T||_{W'}$ is not a  self-map on $W'$.  The multiplicity of $\lambda$ is the dimension of the space  spanned by all the vectors $v\in W'$ such that $Tv=\lambda v$.  By an abuse of notation, we use ${\bf{s}}(T||_{W'})$ to denote the weakly increasing rearrangement of the quasi-eigenvalues $\lambda$ of  $T||_{W'}$, with their multiplicities.

Let $U$ and $V$   be two (finite dimensional) vector spaces. We consider two linear maps $A: U\longrightarrow V$ and $B: V\longrightarrow U$.  Then the nonzero eigenvalues of $AB$ and $BA$ are the same, with same multiplicities (cf. \cite[p. 308]{adv1}).  Let $E_\lambda(AB)$ and $E_\lambda(BA)$ denote the eigenspaces of $AB$ and $BA$ respectively, corresponding to a nonzero eigenvalue $\lambda$.  The isomorphism between $E_\lambda(AB)$ and $E_\lambda(BA)$ is given by
\begin{eqnarray}
&F: E_\lambda(AB)\longrightarrow E_\lambda(BA), \nonumber\\
&F(x)=Bx, ~~~ F^{-1} y = \dfrac{1}{\lambda} Ay.  \label{eq-5.888}
\end{eqnarray}

The next proposition (a), (b) and (c) follow  from a similar argument of \cite[p. 308, (2.6)]{adv1}, and (d) and (e) follow with the help of (\ref{eq-5.888}).

\begin{proposition} \label{pr-5.2}
Let $(\mathcal{H},\phi)$ be a weighted hypergraph and $n\geq 0$. Then 
\begin{enumerate}[(a).]
\item
$ {\bf{s}}\big( (L_n^{\text{Inf}^\phi_*(\mathcal{H}),\phi})^{\text{up} }\big)\overset{\circ}{=} {\bf{s}}\big( (L_{n-1}^{\text{Inf}^\phi_*(\mathcal{H}),\phi})^{\text{down} }\big)$, 
\item
$ {\bf{s}}\big( (L_n^{\text{Sup}^\phi_*(\mathcal{H}),\phi})^{\text{up} }\big)\overset{\circ}{=} {\bf{s}}\big( (L_{n-1}^{\text{Sup}^\phi_*(\mathcal{H}),\phi})^{\text{down} }\big)$, 
\item
${\bf{s}}\big( (L_n^{\Delta\mathcal{H},\phi})^{\text{up} }\big)\overset{\circ}{=} {\bf{s}}\big( (L_{n-1}^{\Delta\mathcal{H},\phi})^{\text{down} }\big)$,
\item
 ${\bf{s}}\big( (L_n^{\Delta\mathcal{H},\phi})^{\text{up} }||_{\text{Inf}_n^\phi(\mathcal{H})}\big)\overset{\circ}{=} {\bf{s}}\big( (L_{n-1}^{\Delta\mathcal{H},\phi})^{\text{down} }||_{\partial^\phi_n\text{Inf}^\phi_n(\mathcal{H})}\big)$,
 \item
 ${\bf{s}}\big( (L_n^{\Delta\mathcal{H},\phi})^{\text{up} }||_{\text{Sup}_n^\phi(\mathcal{H})}\big)\overset{\circ}{=} {\bf{s}}\big( (L_{n-1}^{\Delta\mathcal{H},\phi})^{\text{down} }||_{\partial_n^\phi\text{Sup}_n^\phi(\mathcal{H})}\big)$. 
\end{enumerate}
\end{proposition}
\begin{proof}
We omit the proofs of (a) - (c). We give the proofs of (d) and (e).  Without loss of generality, we assume $n\geq 1$.   In (\ref{eq-5.888}), we consider the two vector spaces $U=\mathbb{F}(\Delta\mathcal{H})_n$, $V=\mathbb{F}(\Delta\mathcal{H})_{n-1}$ and the two linear maps $A=\partial^\phi_n$, $B= (\partial^\phi_n)^*$.  Then $(L_n^{\Delta\mathcal{H},\phi})^{\text{up} }=BA$,  $(L_{n-1}^{\Delta\mathcal{H},\phi})^{\text{down} }=AB$.  For any $\lambda\in \mathbb{F}$ and any $v\in \mathbb{F}(\Delta\mathcal{H})_n$, 
\begin{eqnarray}\label{equiv-5.1}
& \text{\ \ \ \ }&    (L_n^{\Delta\mathcal{H},\phi})^{\text{up} } v=\lambda v\nonumber\\
&\Longleftrightarrow& (L_{n-1}^{\Delta\mathcal{H},\phi})^{\text{down} } \big(\dfrac{1}{\lambda} \partial_n^\phi v \big)=\lambda \big(\dfrac{1}{\lambda} \partial_n^\phi v \big). 
\end{eqnarray}
On the other hand,  if $ (L_n^{\Delta\mathcal{H},\phi})^{\text{up} } v=\lambda v$, then
\begin{eqnarray}\label{equiv-5.2}
 & \text{\ \ \ \ }&  v\in \text{Inf}_n^\phi(\mathcal{H})\nonumber\\
& \Longleftrightarrow & \dfrac{1}{\lambda} \partial_n^\phi v\in \partial_n^\phi  \text{Inf}_n^\phi(\mathcal{H}).  
\end{eqnarray}
Moreover, for any $k\geq 1$,  if $(L_n^{\Delta\mathcal{H},\phi})^{\text{up} } v_i=\lambda v_i$   for each $1\leq i\leq k$, then  
\begin{eqnarray}\label{equiv-5.3}
&\text{\ \ \ \ }&  v_1,v_2,\ldots,v_k 
  \text{    are linearly independent }\nonumber\\
  &\Longleftrightarrow& \dfrac{1}{\lambda} \partial_n^\phi v_1, \dfrac{1}{\lambda} \partial_n^\phi v_2, \ldots, \dfrac{1}{\lambda} \partial_n^\phi v_k 
\text{    are linearly independent}.  
\end{eqnarray}
Hence by (\ref{equiv-5.1}) and (\ref{equiv-5.2}),  for any $\lambda\in \mathbb{F}$,
\begin{eqnarray*}
&\text{\ \ \ \ }&\lambda \text{ is a quasi-eigenvalue of } (L_n^{\Delta\mathcal{H},\phi})^{\text{up} }||_{\text{Inf}_n^\phi(\mathcal{H})}\\
&\Longleftrightarrow& \text{ there exists }   v\in \text{Inf}_n^\phi(\mathcal{H}) \text{ such that }  (L_n^{\Delta\mathcal{H},\phi})^{\text{up} } v=\lambda v\\
&\Longleftrightarrow& \text{ there exists }   \dfrac{1}{\lambda} \partial_n^\phi v\in \partial_n^\phi  \text{Inf}_n^\phi(\mathcal{H}) \text{ such that }  \\
&\text{\ \ \ \ }&(L_{n-1}^{\Delta\mathcal{H},\phi})^{\text{down} } \big(\dfrac{1}{\lambda} \partial_n^\phi v \big)=\lambda \big(\dfrac{1}{\lambda} \partial_n^\phi v \big)\\
&\Longleftrightarrow&  \lambda \text{ is a quasi-eigenvalue of } (L_{n-1}^{\Delta\mathcal{H},\phi})^{\text{down} }||_{\partial^\phi_n\text{Inf}^\phi_n(\mathcal{H})}. 
\end{eqnarray*}
By  (\ref{equiv-5.2}) and (\ref{equiv-5.3}),  the multiplicity of $\lambda$ as a quasi-eigenvalue of  $(L_n^{\Delta\mathcal{H},\phi})^{\text{up} }||_{\text{Inf}_n^\phi(\mathcal{H})}$ equals to the multiplicity of $\lambda$ as a quasi-eigenvalue of  $(L_{n-1}^{\Delta\mathcal{H},\phi})^{\text{down} }||_{\partial^\phi_n\text{Inf}^\phi_n(\mathcal{H})}$. Thus (d) follows. 

 By replacing $\text{Inf}^\phi_*(\mathcal{H})$ with $\text{Sup}^\phi_*(\mathcal{H})$ in the proof of (d), the assertion (e) can be proved similarly. 
\end{proof}

With the help of Proposition~\ref{pr-5.2}, we have the two dimensional case of Proposition~\ref{pr-5.1} in the next corollary.

\begin{corollary}\label{co-5.100}
 Let $(\mathcal{H},\phi)$ be a weighted hypergraph and $n\geq 0$. Suppose the dimensions of the hyperedges of $\mathcal{H}$ are at most $2$.  Then 
\begin{enumerate}[(a).]
\item
${\bf{s}}\big( L_1^{\text{Inf}^\phi_*(\mathcal{H}),\phi}\big) \overset{\circ}{=}{\bf{s}}\big( L_0^{\text{Inf}^\phi_*(\mathcal{H}),\phi}\big)\overset{\circ}{\cup} {\bf{s}}\big( L_2^{\text{Inf}^\phi_*(\mathcal{H}),\phi}\big)$, 
\item
${\bf{s}}\big( L_1^{\text{Sup}^\phi_*(\mathcal{H}),\phi}\big) \overset{\circ}{=}{\bf{s}}\big( L_0^{\text{Sup}^\phi_*(\mathcal{H}),\phi} \big)\overset{\circ}{\cup} {\bf{s}}\big( L_2^{\text{Sup}^\phi_*(\mathcal{H}),\phi} \big)$, 
\item
${\bf{s}}\big( L_1^{\Delta\mathcal{H},\phi}\big) \overset{\circ}{=}{\bf{s}}\big( L_0^{\Delta\mathcal{H},\phi} \big)\overset{\circ}{\cup} {\bf{s}}\big( L_2^{\Delta\mathcal{H},\phi} \big)$. 
\end{enumerate}
\end{corollary} 
\begin{proof}
We notice that $ (L_0^{\text{Inf}^\phi_*(\mathcal{H}),\phi})^\text{up}$, $ (L_0^{\text{Sup}^\phi_*(\mathcal{H}),\phi})^\text{up}$ and $(L_0^{\Delta\mathcal{H},\phi})^\text{up}$ are all zero. And  $ (L_2^{\text{Inf}^\phi_*(\mathcal{H}),\phi})^\text{down}$, $ (L_2^{\text{Sup}^\phi_*(\mathcal{H}),\phi})^\text{down}$ and $(L_2^{\Delta\mathcal{H},\phi})^\text{down}$ are all zero.  The assertions (a), (b) and (c) follow from Proposition~\ref{pr-5.1}~(a)  and  Proposition~\ref{pr-5.2}~(a), Proposition~\ref{pr-5.1}~(b)  and  Proposition~\ref{pr-5.2}~(b), Proposition~\ref{pr-5.1}~(c)  and  Proposition~\ref{pr-5.2}~(c) respectively. 
\end{proof}

The next corollary is a generalization of Corollary~\ref{co-5.100}~(a) and (b).  

\begin{corollary}
 Let $(\mathcal{H},\phi)$ be a weighted hypergraph and $n\geq 0$.  
 \begin{enumerate}[(a).]
 \item
Suppose in $\mathcal{H}$, there are no hyperedges  of dimensions $n -1$ or   $n+3$. Then 
 \begin{eqnarray*}
  {\bf{s}}\big( L_{n+1}^{\text{Inf}^\phi_*(\mathcal{H}),\phi}\big) \overset{\circ}{=}{\bf{s}}\big( L_n^{\text{Inf}^\phi_*(\mathcal{H}),\phi}\big)\overset{\circ}{\cup} {\bf{s}}\big( L_{n+2}^{\text{Inf}^\phi_*(\mathcal{H}),\phi}\big). 
 \end{eqnarray*}
 \item
 Suppose in $\mathcal{H}$, there are no hyperedges  of dimensions $n -1$, $n$,  $n+3$ or $n+4$. Then 
 \begin{eqnarray*}
  {\bf{s}}\big( L_{n+1}^{\text{Sup}^\phi_*(\mathcal{H}),\phi}\big) \overset{\circ}{=}{\bf{s}}\big( L_n^{\text{Sup}^\phi_*(\mathcal{H}),\phi}\big)\overset{\circ}{\cup} {\bf{s}}\big( L_{n+2}^{\text{Sup}^\phi_*(\mathcal{H}),\phi}\big). 
 \end{eqnarray*}
\end{enumerate}
\end{corollary}
\begin{proof}
(a).  We notice that $\text{Inf}_{n-1}^\phi(\mathcal{H})$ and $\text{Inf}^\phi_{n+3}(\mathcal{H})$ are both zero.  Hence   $ (L_n^{\text{Inf}^\phi_*(\mathcal{H}),\phi})^\text{up}$ and $ (L_{n+2}^{\text{Inf}^\phi_*(\mathcal{H}),\phi})^\text{down}$ are both zero.  The corollary follows from Proposition~\ref{pr-5.1}~(a)  and  Proposition~\ref{pr-5.2}~(a).  

(b).  By replacing $\text{Inf}^\phi_*(\mathcal{H})$ with $\text{Sup}^\phi_*(\mathcal{H})$ in the proof of (a), the assertion (b) can be proved similarly using  
Proposition~\ref{pr-5.1}~(b)  and  Proposition~\ref{pr-5.2}~(b). 
\end{proof}

The next proposition is a consequence of Lemma~\ref{le-linearalg}. 

\begin{proposition}\label{pr-5.3}
  Let $(\mathcal{H},\phi)$ be a weighted hypergraph and $n\geq 0$. Then 
  \begin{enumerate}[(a).]
  \item
  ${\bf{s}}\big( (L_n^{\Delta\mathcal{H},\phi})^{\text{up} }||_{\text{Inf}_n^\phi(\mathcal{H})}\big)\subseteq {\bf{s}}\big( (L_n^{\text{Inf}^\phi_n(\mathcal{H}),\phi})^{\text{up} }\big)$, 
  \item
   ${\bf{s}}\big( (L_n^{\Delta\mathcal{H},\phi})^{\text{up} }||_{\text{Sup}_n^\phi(\mathcal{H})}\big)\subseteq {\bf{s}}\big( (L_n^{\text{Sup}^\phi_n(\mathcal{H}),\phi})^{\text{up} }\big)$. 
  \end{enumerate}
\end{proposition}
\begin{proof}
Without loss of generality, we assume $n\geq 1$. 
By Lemma~\ref{le-linearalg}, we have the following commutative diagram
\begin{eqnarray*}
\xymatrix{
\text{Inf}_{n-1}^{\phi}(\mathcal{H}) \ar[rrrdd]_{(\partial_n^\phi\mid_{\text{Inf}_*^\phi(\mathcal{H})})^*~~~~}\ar[rrr]^{(\partial_n^\phi)^*\mid_{\text{Inf}_{n-1}^\phi(\mathcal{H})}} &&& \mathbb{F}(\Delta\mathcal{H})_n\ar[dd]^{\text{orthogonal proj.}}_{p}\\
\\
&&&\text{Inf}_n^\phi(\mathcal{H}). 
}
\end{eqnarray*}
By the commutative diagram,  
\begin{eqnarray*}
  (L_n^{\text{Inf}^\phi_*(\mathcal{H}),\phi})^{\text{up} } = p\circ \big( (L_n^{\Delta\mathcal{H},\phi})^{\text{up} }||_{\text{Inf}_n^\phi(\mathcal{H})}\big).  
  \end{eqnarray*}
   Hence for any    quasi-eigenvalue $\lambda$ of $(L_n^{\Delta\mathcal{H},\phi})^{\text{up} }||_{\text{Inf}_*^\phi(\mathcal{H})}$, $\lambda$ is also an eigenvalue of $(L_n^{\text{Inf}^\phi_*(\mathcal{H}),\phi})^{\text{up} } $. And the multiplicity of $\lambda$ as a quasi-eigenvalue of  $(L_n^{\Delta\mathcal{H},\phi})^{\text{up} }||_{\text{Inf}_*^\phi(\mathcal{H})}$ is smaller than or equal to the multiplicity of $\lambda$ as an eigenvalue of  $(L_n^{\text{Inf}^\phi_*(\mathcal{H}),\phi})^{\text{up} } $. Thus (a) follows. 
   
   By replacing $\text{Inf}^\phi_*(\mathcal{H})$ with $\text{Sup}^\phi_*(\mathcal{H})$ in the proof of (a), the assertion (b) can be proved similarly. 
\end{proof}

The next theorem follows from Proposition~\ref{pr-5.1}, Proposition~\ref{pr-5.2} and Proposition~\ref{pr-5.3}. 

\begin{theorem}\label{th-spec}
  Let $(\mathcal{H},\phi)$ be a weighted hypergraph and $n\geq  0$. Then 
\begin{enumerate}[(a).]
\item
 ${\bf{s}}\big( (L_n^{\Delta\mathcal{H},\phi})||_{\partial_{n+1}^\phi\text{Inf}_{n+1}^\phi(\mathcal{H})}\big)\overset{\circ}{\subseteq} {\bf{s}}\big(L_n^{\text{Inf}^\phi_*(\mathcal{H}),\phi}\big)$, 
  \item
   ${\bf{s}}\big( (L_n^{\Delta\mathcal{H},\phi}) ||_{\partial_{n+1}^\phi\text{Sup}_{n+1}^\phi(\mathcal{H})}\big)\overset{\circ}{\subseteq} {\bf{s}}\big(L_n^{\text{Sup}^\phi_*(\mathcal{H}),\phi}\big)$. 
\end{enumerate}
\end{theorem}

\begin{proof}
The assertion (a) follows from the calculation
\begin{eqnarray}\label{eq-5.0}
{\bf{s}}\big(L_n^{\text{Inf}^\phi_*(\mathcal{H}),\phi}\big)&\overset{\circ}{=}& {\bf{s}}\big( (L_n^{\text{Inf}^\phi_*(\mathcal{H}),\phi})^{\text{up} }\big)\overset{\circ}{\cup} {\bf{s}}\big( (L_n^{\text{Inf}^\phi_*(\mathcal{H}),\phi})^{\text{down} }\big)\nonumber\\
&\overset{\circ}{=}& {\bf{s}}\big( (L_n^{\text{Inf}^\phi_*(\mathcal{H}),\phi})^{\text{up} }\big)\overset{\circ}{\cup} {\bf{s}}\big( (L_{n+1}^{\text{Inf}^\phi_*(\mathcal{H}),\phi})^{\text{up} }\big) \nonumber\\
&\supseteq&{\bf{s}}\big( (L_n^{\Delta\mathcal{H},\phi})^{\text{up} }||_{\text{Inf}_n^\phi(\mathcal{H})}\big) \overset{\circ}{\cup} {\bf{s}}\big( (L_{n+1}^{\Delta\mathcal{H},\phi})^{\text{up} }||_{\text{Inf}_{n+1}^\phi(\mathcal{H})}\big)\nonumber\\
& \overset{\circ}{=}&{\bf{s}}\big( (L_n^{\Delta\mathcal{H},\phi})^{\text{up} }||_{\text{Inf}_n^\phi(\mathcal{H})}\big) \overset{\circ}{\cup} {\bf{s}}\big( (L_{n}^{\Delta\mathcal{H},\phi})^{\text{down} }||_{\partial^\phi_{n+1}\text{Inf}_{n+1}^\phi(\mathcal{H})}\big)\nonumber\\
& \supseteq&{\bf{s}}\big( (L_n^{\Delta\mathcal{H},\phi})^{\text{up} }||_{\partial_{n+1}^\phi\text{Inf}_{n+1}^\phi(\mathcal{H})}\big) \overset{\circ}{\cup} {\bf{s}}\big( (L_{n}^{\Delta\mathcal{H},\phi})^{\text{down} }||_{\partial^\phi_{n+1}\text{Inf}_{n+1}^\phi(\mathcal{H})}\big)\nonumber\\
&=& {\bf{s}}\big( (L_n^{\Delta\mathcal{H},\phi})||_{\partial^\phi_{n+1}\text{Inf}^\phi_{n+1}(\mathcal{H})}\big). 
\end{eqnarray}

 By replacing $\text{Inf}^\phi_*(\mathcal{H})$ with $\text{Sup}^\phi_*(\mathcal{H})$ in the proof of (a),  the assertion (b) can be proved similarly. 
\end{proof}

\begin{remark}
Since $\partial^\phi_{n}\partial^\phi_{n+1}=0$, in the fifth line of (\ref{eq-5.0}), 
\begin{eqnarray*}
{\bf{s}}\big( (L_n^{\Delta\mathcal{H},\phi})^{\text{up} }||_{\partial_{n+1}^\phi\text{Inf}_{n+1}^\phi(\mathcal{H})}\big)=\{0,\ldots,0\}. 
\end{eqnarray*}
\end{remark}

As special cases of Theorem~\ref{th-spec}, the next corollary follows from Lemma~\ref{le-linearalg} and the proof of Theorem~\ref{th-spec}.

\begin{corollary}
  Let $(\mathcal{H},\phi)$ be a weighted hypergraph and $n\geq  0$.   Let the spaces $A_n^\phi$, $B_n^\phi$ and $E_n^\phi$ be given by (\ref{eq-4q1}), (\ref{eq-4q2}) and (\ref{eq-4q3}). 
  \begin{enumerate}[(a).]
  \item 
  If    $\partial_n(A^\phi_{n+1}\oplus B^\phi_{n+1}\oplus E^\phi_{n+1})\subseteq A^\phi_{n}\oplus B^\phi_{n}\oplus E^\phi_{n}$, then
  \begin{eqnarray*}
  {\bf{s}}\big(L_n^{\text{Inf}^\phi_*(\mathcal{H}),\phi}\big)\overset{\circ}{=}{\bf s}\big( (L_n^{\Delta\mathcal{H},\phi})^{\text{up} }||_{\text{Inf}_n^\phi(\mathcal{H})}\big) \overset{\circ}{\cup} {\bf{s}}\big( (L_{n}^{\Delta\mathcal{H},\phi})^{\text{down} }||_{\partial^\phi_{n+1}\text{Inf}_{n+1}^\phi(\mathcal{H})}\big);
  \end{eqnarray*}
  \item 
  If $\partial_n( E^\phi_{n+1})\subseteq E^\phi_{n}$,  then 
   \begin{eqnarray*}
  {\bf{s}}\big(L_n^{\text{Sup}^\phi_*(\mathcal{H}),\phi}\big)\overset{\circ}{=}{\bf s}\big( (L_n^{\Delta\mathcal{H},\phi})^{\text{up} }||_{\text{Sup}_n^\phi(\mathcal{H})}\big) \overset{\circ}{\cup} {\bf{s}}\big( (L_{n}^{\Delta\mathcal{H},\phi})^{\text{down} }||_{\partial^\phi_{n+1}\text{Sup}_{n+1}^\phi(\mathcal{H})}\big). 
  \end{eqnarray*}
\end{enumerate}
\end{corollary}

\begin{proof}
(a). Suppose $\partial_n(A^\phi_{n+1}\oplus B^\phi_{n+1}\oplus E^\phi_{n+1})\subseteq A^\phi_{n}\oplus B^\phi_{n}\oplus E^\phi_{n}$.  Then by Lemma~\ref{le-linearalg}, 
\begin{eqnarray*}
(\partial_{n+1}^\phi\mid_{\text{Inf}_*^\phi(\mathcal{H})})^*= (\partial_{n+1}^\phi)^*\mid_{\text{Inf}_{n}^\phi(\mathcal{H})}. 
\end{eqnarray*}
By the proof of Proposition~\ref{pr-5.3},
\begin{eqnarray*}
  (L_{n+1}^{\text{Inf}^\phi_*(\mathcal{H}),\phi})^{\text{up} } =  (L_{n+1}^{\Delta\mathcal{H},\phi})^{\text{up} }||_{\text{Inf}_{n+1}^\phi(\mathcal{H})}.  
  \end{eqnarray*} 
  With the help of the third  and forth line of (\ref{eq-5.0}), we obtain (a). 
  
   (b). Suppose $\partial_n( E^\phi_{n+1})\subseteq E^\phi_{n}$.  Then by Lemma~\ref{le-linearalg}, 
\begin{eqnarray*}
(\partial_{n+1}^\phi\mid_{\text{Sup}_*^\phi(\mathcal{H})})^*= (\partial_{n+1}^\phi)^*\mid_{\text{Sup}_{n}^\phi(\mathcal{H})}. 
\end{eqnarray*}
 By replacing $\text{Inf}^\phi_*(\mathcal{H})$ with $\text{Sup}^\phi_*(\mathcal{H})$ in the proof of (a),  the assertion (b) can be proved similarly. 
\end{proof}

\section{Discussions of Hypergraphs and Paths on Digraphs}\label{sec6}

In this section, we discuss the relations between hypergraphs and paths on digraphs. 

\smallskip

\begin{definition}\cite{yau1}
A digraph $G$ is a pair $(V,E)$ where $V$ is a set (called the vertex set) and $E$ is a subset of $V\times V$. If $(a,b)\in E$, then $(a,b)$ is called a directed edge, and is denoted as $a\to b$.  
\end{definition}

Let $V$ be a non-empty set and $G=(V,E)$ be a digraph.

\begin{definition}\cite[Definition~2.1, Example~3.3 and page 19]{yau1}
An elementary $n$-path (or an elementary path of length $n$) on $V$  is a sequence  $v_0v_1\ldots v_n$ where $v_0,v_1,\ldots,v_n\in V$. Here $v_0,v_1,\ldots,v_n$ are not required to be distinct.  An allowed elementary $n$-path (or an allowed elementary path of length $n$) on $G$ is an elementary $n$-path $v_0 v_1\ldots v_n$ on $V$ such that for each $i\geq 1$, $v_{i-1}\to v_i$ is a directed edge of $G$.  An alllowed elementary $n$-path $v_0v_1\ldots v_n$ is called closed   if $v_0=v_n$.  
\end{definition}

\begin{itemize} 
\item
 {\it  A digraph  without closed allowed elementary  paths gives  a hypergraph     }
 \end{itemize}
 
 Let $G$ be a digraph without closed allowed elementary  paths.  Then for each $a,b\in V$, at most one of $(a,b)$ and $(b,a)$ is a directed edge of $G$. We write $a\prec b$ if either $a\to b$  or there exists $n\geq 0$ and $v_0,v_1\ldots, v_n\in V$ such that $a v_0 v_1\ldots v_n b$ is a path on $G$.  Then $a\prec b$ and $b\prec c$ imply $a\prec c$. Hence equipped with the relation $\prec $, the set $V$ is a partially ordered set.    
 For each $n\geq 0$, let  $\mathcal{H}_n$ be the set of all allowed elementary $n$-paths on  $G$.  In particular, $\mathcal{H}_0=V$. Let $\mathcal{H}= \sqcup_{n=0}^\infty \mathcal{H}_n$.  
 For any  $\sigma\in \mathcal{H}$,  the relation $\prec$ gives a total order on the set of the vertices of $\sigma$.    The next lemma  follows. 
 
 \begin{lemma}\label{le-6.1}
 For any $\sigma,\tau\in \mathcal{H}$, $\sigma$ and $\tau$ are distinct allowed elementary  paths on $G$ if and only if as subsets of $V$, $\sigma\neq \tau$. 
 \qed
 \end{lemma}
 
 The next proposition follows from Lemma~\ref{le-6.1}.
 
 \begin{proposition}\label{pr-6.1}
  Let $G$ be a digraph without closed allowed elementary  paths. Then the collection $\mathcal{H}$ of all allowed elementary paths on $G$ is a hypergraph.  The boundary map $\partial_*$ of $\Delta\mathcal{H}$ is given by 
  \begin{eqnarray*}
  \partial_n (v_0v_1\ldots v_n) = \sum_{i=0}^n (-1)^i v_0\ldots \hat{v_i}\ldots v_n, 
  \end{eqnarray*}
  which coincides with the restriction of the boundary map in \cite[Definition~2.3]{yau1} to $\Delta\mathcal{H}$. 
 \qed
 \end{proposition} 

 If the digraph $G$ has closed allowed elementary  paths, then it may happen that certain vertices repeat in an allowed elementary path on $G$,  and two distinct allowed elementary  paths on $G$ are the same as subsets of $V$. In this case,  Lemma~\ref{le-6.1} and Proposition~\ref{pr-6.1} do not hold.

\begin{itemize}
\item  \it A hypergraph  gives a weighted digraph   
\end{itemize}

Let $\mathcal{H}$ be a hypergraph.  Consider the digraph $G^\mathcal{H}$ whose set of vertices is   $\mathcal{H}$, and whose set of  edges is defined as follows: for any $\sigma, \tau\in\mathcal{H}$, $\sigma\to \tau$ if and only if $\sigma\supset \tau$ and $\sigma\neq \tau$.  In particular, when $\mathcal{H}$ is a simplicial complex, the construction of $G^\mathcal{H}$ is given in \cite[Example~3.8]{yau1}.  

A weighted digraph  $(G,w)$ is obtained by  assigning a value $w(a\to b)$ to each directed edge $a\to b$ on a digraph $G$.  In the digraph $G^\mathcal{H}$, we assign the value 
\begin{eqnarray*}
w(\sigma\to\tau)=\text{card}(\sigma\setminus \tau)=\dim \sigma-\dim\tau 
\end{eqnarray*}
to each directed edge $\sigma\to \tau$.  Then we obtain a weighted digraph.

\begin{itemize}
\item \it A hypergraph  gives   a pair of a digraph and a subset of the vertex set of the digraph    
\end{itemize}

   Let $\mathcal{K}$  be a simplicial complex.  Let $V$ be the set of all simplices of $\mathcal{K}$. For any $n\geq 0$ and any  $n$-simplex $\sigma\in\mathcal{K}$, we let $\sigma\to d_i\sigma$ be a directed edge for each  $0\leq i\leq n$.  We obtain a digraph $G_\mathcal{K}$.   The digraph $G_\mathcal{K}$ has no closed allowed elementary  paths.  Nevertheless, a digraph  with no closed allowed elementary  paths may not be able to be realized as $G_\mathcal{K}$.

For a hypergraph $\mathcal{H}$, we consider the digraph $G_{\Delta\mathcal{H}}$.   The vertex set of $G_{\Delta\mathcal{H}}$ is the  set of simplices of $\Delta\mathcal{H}$.  Hence $\mathcal{H}$ is a subset of the vertex set of $G_{\Delta\mathcal{H}}$. 
Therefore, $\mathcal{H}$ can be represented by a pair $(G_{\Delta\mathcal{H}}, U_\mathcal{H})$, where $U_\mathcal{H}$ is a subset of the vertex set of $G_{\Delta\mathcal{H}}$.  
The next proposition follows. 

\begin{proposition}
A hypergraph $\mathcal{H}$ gives a pair $(G_{\Delta\mathcal{H}}, U_\mathcal{H})$, where $G_{\Delta\mathcal{H}}$ is a digraph and  $U_\mathcal{H}$ is a subset of the vertex set of $G_{\Delta\mathcal{H}}$.  
 \qed
\end{proposition}

 \bigskip

\noindent {\bf Acknowledgement}.  The project was supported in part by the Singapore Ministry of Education research grant
(AcRF Tier 1 WBS No. R-146-000-222-112). The first author was supported   by Guangdong Ocean University. The second author was supported in part by the
President's Graduate Fellowship of National University of Singapore. The
third author was supported by a grant (No. 11329101) of NSFC of China.

 \bigskip

 Shiquan Ren (for correspondence) 
 
 Address:   $^a$ School of Mathematics and Computer Science, Guangdong Ocean University. Haida Road 1, Zhanjiang 524088, China. $^b$ Department of Mathematics, National University of Singapore. 119076, Singapore.

  e-mail: sren@u.nus.edu
  
  \bigskip
  
Chengyuan Wu
  
Address: Department of Mathematics, National University of Singapore. 119076, Singapore. 
  
  e-mail: wuchengyuan@u.nus.edu
  
  \bigskip
  
Jie Wu
  
Address: Department of Mathematics, National University of Singapore. 119076, Singapore. 
  
  e-mail: matwuj@nus.edu.sg


\begin{thebibliography}{99}


\bibitem{anderson}
W.N. Anderson and T.D. Morley, \emph{Eigenvalues of the Laplacian  of a graph}. Univ. of Maryland Tech. Report, {\bf TR-71-45} (1971); Linear Multilinear A. {\bf 18} (1985), 141-145. 


\bibitem{ban}
A. Banerjee and J. Jost, \emph{On the spectrum of the normalized graph Laplacian}. Linear Algebra Appl. {\bf 428} (2008),  3015-3022. 

\bibitem{berge}
 C. Berge,  \emph{Graphs and hypergraphs}. North-Holland Mathematical Library, Amsterdam, 1973.  







\bibitem{h1}
S. Bressan, J. Li, S. Ren and J. Wu, \emph{ The embedded homology of hypergraphs and applications}. Asian J. Math. accepted (2018).  https://arxiv.org/abs/1610.00890. 




\bibitem{chung123}
F.R.K. Chung, \emph{The Laplacian of a hypergraph}. DIMACS Ser. in Discrete Math. Theoret. Comput. Sci. {\bf 10} (1983), 21-36. 

\bibitem{chungbook1}
F.R.K. Chung, \emph{Spectral graph theory}. CBMS Reg. Conf. Ser. in Math. {\bf 92}, Amer. Math. Soc., Providence, RI, 1997. 





\bibitem{cve}
D.M. Cvetkovi\'c, M. Doob and H. Sachs, \emph{Spectra of graphs. theory and applications}. 3rd ed., Johann Ambrosius Barthm Heidelberg, 1995. 





\bibitem{daw}
R.J. MacG. Dawson, \emph{ Homology of weighted simplicial complexes}. Cahiers de Topologie et G\'{e}om\'{e}trie Diff\'{e}rentielle Cat\'{e}goriques {\bf 31} (3) (1990), 229-243. 

\bibitem{duv}
A.M.  Duval and V. Reiner, \emph{Shifted simplicial complexes are Laplacian integral}. Trans. Amer. Math. Soc. {\bf 354} (2002), 4313-4344. 

\bibitem{eck}
B. Eckmann, \emph{Harmonische funktionen und fandwertaufgaben in einem komplex}. Comment. Math. Helv. {\bf 17} (1) (1944), 240-255. 















\bibitem{yau1}
A. Grigor'yan, Y. Lin, Y. Muranov and S.T. Yau, \emph{Homologies of path complexes and digraphs}. arXiv (2012).  http://arxiv.org/abs/1207.2834. 



 


  \bibitem{glob1}
   D. Horak and J. Jost, \emph{Interlacing inequalities for eigenvalues of discrete Laplace operators}. Ann. Global Anal. Geom. {\bf 43} (2) (2013),  177-207. 

  \bibitem{adv1}
  D. Horak and J. Jost, \emph{Spectra of combinatorial Laplace operators on simplicial complexes}. Adv. Math. {\bf 244} (2013), 303-336. 
  
  \bibitem{huqi}
S. Hu and L. Qi,  \emph{  The Laplacian of a uniform hypergraph}. J. Comb. Optim. {\bf  29} (2) (2015),  331-366. 

  
\bibitem{1847}
G. Kirchhoff, \emph{\"Uber de Aufl\"osung der Gleichungen auf welche man bei der Untersuchen der linearen Vertheilung galvanischer Str\"ome gef\"uht wird}. Ann. der Phys. und Chem. {\bf 72} (1847), 495-508. 












 



  
  \bibitem{lio}
  P. Li  and O. Milenkovic, \emph{Submodular hypergraphs: $p$-Laplacians, Cheeger inequalities and
spectral clustering}.  Proceedings of Machine Learning Research {\bf 80}.  
  
\bibitem{morita}
S. Morita, \emph{Geometry of differential forms}. Translations of Mathematical Monographs {\bf 201}, American Mathematical Society, 2001. 
 

  \bibitem{parks}
  A.D. Parks  and S.L. Lipscomb, \emph{Homology and hypergraph acyclicity: a combinatorial invariant for hypergraphs}.   Naval Surface Warfare Center, 1991.  
  

  




 

  
 
 \bibitem{ev}
 S. Ren, C. Wu and J. Wu, \emph{Evolutions of hypergraphs and their embedded homology}. arXiv (2018). https://arxiv.org/abs/1804.07132. 
 
  \bibitem{rocky}
 S. Ren, C. Wu and J. Wu, \emph{Weighted persistent homology}. 
Rocky Mountain J. Math. accepted (2018). https://arxiv.org/abs/1804.07132. 
 
 \bibitem{chengyuan2}
 S. Ren, C. Wu and J. Wu, \emph{Computational tools in weighted persistent homology}.  https://arxiv.org/abs/1711.09211. 
 
 \bibitem{chengyuan}
C. Wu, S. Ren, J. Wu and K. Xia, \emph{Weighted cohomology and weighted Laplacian}. arXiv (2018).  https://arxiv.org/abs/1804.06990. 
 
 
 


 
 
 
 


 
 
 
\end{thebibliography}
\end{document}